\documentclass[a4paper,reqno]{amsart}
\usepackage{latexsym}
\usepackage{amssymb,amsthm,amsmath}
\usepackage{graphicx}
\usepackage{graphics}
\usepackage{epstopdf}
\usepackage{caption}
\usepackage{epsfig}
\usepackage{float}
\theoremstyle{plain}
\newtheorem{theorem}{Theorem}
\newtheorem{lemma}[theorem]{Lemma}
\newtheorem{corollary}[theorem]{Corollary}
\newtheorem{proposition}[theorem]{Proposition}
\newtheorem{conjecture}[theorem]{Conjecture}

\theoremstyle{definition}
\newtheorem{definition}[theorem]{Definition}
\theoremstyle{remark}
\newtheorem{remark}[theorem]{Remark}

\newcommand{\auf}{\left\langle}
\newcommand{\zu}{\right\rangle}
\DeclareMathOperator{\esssup}{esssup}

\def\Rd{{\mathbf{R}^{d}}}

\let\mib=\boldsymbol
\def\mxi{{\mib \xi}}
\def\mx{{\mathbf{x}}}

\def\N{{\mathbf{N}}}

\begin{document}

\title[Microlocal defect functionals in VMO]{Microlocal defect functionals in VMO: Geometric localisation and applications to highly heterogeneous media}

\author{M.~Mi\v{s}ur}
\address{Marin Mi\v{s}ur,
University of Zagreb, Faculty of Science, Bijeni\v{c}ka cesta 30,
10000 Zagreb, Croatia.
E-mail: mmisur@math.hr}

\date{\today}

\begin{abstract}
We extend the concept of microlocal defect functionals to test functions belonging to the space $\mathrm{L}^\infty\cap\mathrm{VMO}_c$.
Following L.~Tartar's remark that an extension of such concepts to $\mathrm{VMO}$ spaces should be possible, we establish a
functional-analytic framework for this extension within the $\mathrm{L}^p-\mathrm{L}^q$ setting. Because the topological dual of
$\mathrm{VMO}$ is the Hardy space $\mathcal{H}^1$, the resulting object takes the form of an H-distribution rather than a non-negative
Radon measure. By assuming strict H\"older conjugate inequalities, we use the John-Nirenberg inequality over localized domains to
construct these functionals. We show that the functional acts as a distribution on the inductive limit topology of the test spaces,
giving geometric localisation principles for equations with rough $\mathrm{VMO}$ coefficients, that is, coefficients admitting sharp
transitions of vanishing mean oscillation. We illustrate the framework in three settings: stratified transport, zero-order non-local
cross-phase energies, and sub-critical acoustic scattering in high-contrast media. These differ in the nonlinearity generating the
companion sequence but share a common geometric conclusion, since in each case the macroscopic energy defect is confined to the
characteristic variety of the underlying flow, being supported where $\sum_j a_j(\mx)\xi_j=0$.
\end{abstract}


\maketitle

\section{Introduction}\label{sec:intro}
The theory of H-measures, introduced independently by Tartar \cite{Tgth} and G\'erard \cite{Gerard}, has become a standard tool in
studying the oscillations and concentration effects of weakly converging sequences in $\mathrm{L}^2$. A fundamental requirement of
this classical construction is the commutation of multiplication and Fourier multiplier operators. Tartar noted in his introductory
text \cite{Tartarbook} that because the commutator $[M_b, \mathcal{A}_\psi]$ remains compact when the multiplication symbol $b$
belongs to the space of functions of vanishing mean oscillation ($\mathrm{VMO}$) \cite{Uchiyama}, an extension of microlocal defect
functionals to $\mathrm{VMO}$ test functions should be possible.

In this note, we establish the mathematical framework required for such an extension. Because the topological dual of global
$\mathrm{VMO}(\Rd)$ is the Hardy space $\mathcal{H}^1(\Rd)$ \cite{SteinHA}, and elements of $\mathcal{H}^1(\Rd)$ must satisfy
a zero-integral moment condition, any non-negative measure acting continuously on global $\mathrm{VMO}(\Rd)$ is necessarily
trivial. While restricting the test space to compactly supported functions shifts the dual to a local Hardy space (bypassing
this strict moment condition), the strict integrability gap of the $\mathrm{L}^p-\mathrm{L}^q$ framework fundamentally
necessitates that the proposed extension takes the form of an H-distribution, joining other significant generalizations of
the classical theory, such as the ultraparabolic H-measures developed by Panov \cite{Panov}. The H-distribution framework
was introduced by Antoni\'c and Mitrovi\'c \cite{AM} for the $\mathrm{L}^p-\mathrm{L}^q$ setting, generalized for compensated
compactness by Mi\v{s}ur and Mitrovi\'c \cite{MM}, and further expanded to encompass distributions of anisotropic
order \cite{AntonicErcegMisur, MisurPalle2025}. We demonstrate that by introducing asymmetry into the integrability of the
weakly converging sequences, we obtain sufficient integrability conditions to apply the John-Nirenberg inequality.
Restricting the spatial test functions to have compact support permits the construction of H-distributions that are continuous
with respect to the inductive limit topology of $\mathrm{VMO}_c$, providing an extension to $\mathrm{VMO}$ coefficients while
respecting the functional-analytic properties of the space on $\Rd$. These three requirements, namely compact support, the strict
integrability gap, and the LF-space topology, are the structural prerequisites underlying the continuity bound; we collect and
discuss them once, in Remark \ref{rem:prerequisites} below, and refer back to them thereafter rather than restating them.

We emphasize what is, and what is not, new here, since the construction shares its architecture with the existing H-distribution
literature. The scheme of Theorem \ref{thm:h_dist_vmo}, namely subsequence extraction, transfer of the multiplier onto the second
sequence by its adjoint, elimination of the commutator, and diagonal extraction over an exhausting family of compact sets,
follows \cite{AM, MM}, and the commutator compactness driving it is quoted from \cite{AMM, mmphd} (Lemma \ref{1stcommlemma_Lp}).
The new content lies in replacing the continuous test functions of those works by the rough class
$\mathrm{L}^\infty\cap\mathrm{VMO}_c(\Rd)$, which forces three steps absent from the prior constructions. First, continuity is
no longer a supremum bound but is obtained through the strict gap $\frac1p+\frac1q<1$ and a localized John--Nirenberg inequality
(Proposition \ref{prop:JN_constant}). Second, the resulting estimate is only local, so the functional is continuous for the
inductive-limit topology of the strict LF-space $\mathrm{VMO}_c(\Rd)$, a space not previously introduced, whose requisite
properties are established in the Appendix and invoked at the corresponding steps. Third, in Theorems \ref{thm:localisation}
and \ref{thm:defect_rep} the coefficients are only of class $\mathrm{L}^\infty\cap\mathrm{VMO}$, handled without weak
differentiability via a Riesz-potential testing sequence and the Banach-algebra absorption of the coefficients into the test slot.
Beyond this substitution, Theorems \ref{thm:localisation} and \ref{thm:defect_rep} follow the short arguments of \cite{AM, MM} closely;
the new content is the functional-analytic apparatus that renders the $\mathrm{VMO}_c$ evaluations legitimate.

The paper is structured as follows. In Section \ref{sec:main_result}, we construct the $\mathrm{VMO}$ H-distribution and establish
its bounds via the John-Nirenberg inequality, addressing the topological necessity of the inductive limit space $\mathrm{VMO}_c(\Rd)$.
Section \ref{sec:localisation} presents the Localisation Principle, bypassing the lack of weak differentiability in $\mathrm{VMO}$
coefficients. In Section \ref{sec:defect_rep}, we establish the macroscopic defect representation theorem, followed by the formalization
of canonical testing sequences in Section \ref{sec:canonical}, adapted for unbounded domains, where we also isolate a unified defect
theorem that abstracts the common core shared by the applications. The remainder of the paper applies the framework to several PDE classes.
In Section \ref{sec:transport}, we use the Localisation Principle to determine the characteristic support of stationary transport defects
in stratified fluids. In Sections \ref{sec:zero_order} and \ref{sec:acoustics}, we evaluate zero-order cross-phase energy defects and the
geometric trapping of sub-critical acoustic scattering in high-contrast media. Together, these applications show how the framework
identifies where macroscopic energy is trapped and scattered by $\mathrm{VMO}$ heterogeneities. Finally, an Appendix establishes the
topological properties of the strict LF-space $\mathrm{VMO}_c(\Rd)$. This includes proofs concerning its Fr\'echet duality,
the failure of the Montel property, and its compatibility with fractional Sobolev regularity, which together justify the operations used
throughout the main text.

\section{Notation, function spaces and auxiliary results}\label{sec:notation}
Let $\Omega\subset\Rd$ be a bounded open connected set.
A locally integrable function $f$ on $\Omega$ belongs to the space of \emph{functions of bounded mean oscillation}, 
denoted by $\mathrm{BMO}(\Omega)$, if there exists a constant $A>0$ such that for all open balls $B$ with $\mathrm{Cl}B\subset\Omega$:

\begin{equation}\label{BMOdef}
\frac{1}{|B|}\int_B |f-f_B|\, d\mx\leq A\;,
\end{equation}

\noindent where $f_B$ is the mean value of $f$ over $B$.
Identifying functions that differ by a constant almost everywhere, $\mathrm{BMO}(\Omega)$ is a complete, non-separable Banach space.
Its closed subspace of \emph{functions of vanishing mean oscillation}, denoted by $\mathrm{VMO}(\Omega)$, is defined as the closure of
$\mathrm{C}_c(\Omega)$ in the $\mathrm{BMO}(\Omega)$ norm \cite{CoifW}.
Throughout, $\mathrm{VMO}$ denotes this Coifman--Weiss space (the $\mathrm{C}_c$-closure), and not Sarason's space obtained as the closure
of bounded uniformly continuous functions; this is the choice for which the duality $\mathrm{VMO}(\Rd)' = \mathcal{H}^1(\Rd)$ employed below holds.
The space $\mathrm{VMO}(\Omega)$ is separable and complete. 

To use the necessary compactness properties of our operators on $\Rd$, we require our test functions to have compact support.
For a compact set $K\subset\Rd$, we define the space $\mathrm{VMO}_K(\Rd)$ as the space of all $\mathrm{VMO}(\Rd)$ functions with compact support in $K$:

\begin{equation*}
\mathrm{VMO}_K(\Rd) = \left\{ \varphi\in\mathrm{VMO}(\Rd):\, \mathrm{supp}\varphi\subseteq K \right\}.
\end{equation*}

Because convergence in the global $\mathrm{VMO}(\Rd)$ norm implies convergence in $\mathrm{L}^1_{loc}(\Rd)$, the support of the limit
function is forced to remain almost everywhere within $K$. Therefore, $\mathrm{VMO}_K(\Rd)$ is a closed subspace of the complete Banach
space $\mathrm{VMO}(\Rd)$, making it a Banach space in its own right.
Given an increasing sequence of compact sets $(K_n)_n$ exhausting $\Rd$, i.e., $K_i\subset \mathrm{Int}K_{i+1}$ and $\Rd = \bigcup_{n\in\N}K_n$,
we consider the space of all $\mathrm{VMO}(\Rd)$ functions with compact support:
\begin{equation*}
\mathrm{VMO}_c(\Rd) = \bigcup_{n\in\N} \mathrm{VMO}_{K_n}(\Rd)\, .
\end{equation*}

We equip $\mathrm{VMO}_c(\Rd)$ with the locally convex \emph{inductive limit topology} generated by the natural inclusions
$i_n:\mathrm{VMO}_{K_n}(\Rd)\to \mathrm{VMO}_c(\Rd)$. 
Because each step space $\mathrm{VMO}_{K_n}(\Rd)$ is a closed subspace of $\mathrm{VMO}_{K_{n+1}}(\Rd)$, this forms a
\emph{strict inductive limit}, establishing $\mathrm{VMO}_c(\Rd)$ as a strict LF-space. 
Consequently, the space is complete, Hausdorff, and the topology it induces on any subspace $\mathrm{VMO}_{K_n}(\Rd)$ coincides
with its original Banach space topology.
By the universal property of strict LF-spaces, verifying the continuity of a linear or bilinear functional on $\mathrm{VMO}_c(\Rd)$
reduces entirely to verifying its continuity on each fixed Banach space $\mathrm{VMO}_{K_n}(\Rd)$. This structural guarantee justifies
bounding our functionals in subsequent sections strictly via the standard $\mathrm{BMO}$ norm localized to the specific compact support
of the test functions.

This extrinsic topological construction is a deliberate and necessary choice. If one were to define the space intrinsically, evaluating
the mean oscillation only over balls strictly contained within the compact set $K$, then extending such functions by zero to the whole
of $\Rd$ would introduce sharp boundary jumps. Because of the non-local nature of the John-Nirenberg mean oscillation, evaluating the
$\mathrm{BMO}$ norm across these artificial boundaries causes the integrals to blow up, mirroring the well-known zero-extension pathologies
of fractional Sobolev spaces on bounded domains. 

By defining $\mathrm{VMO}_c(\Rd)$ extrinsically, via the strict inductive limit of global functions restricted to compact supports, the
functions possess a controlled decay to zero. This bypasses the boundary singularities and lets non-local operators, such as global
Fourier multipliers and fractional derivatives, act on the space without restriction.

The ambient test space is the full strict LF-space $\mathrm{VMO}_c(\Rd)$, endowed with the inductive limit topology generated by the
$\mathrm{BMO}$ norms of its Banach steps $\mathrm{VMO}_{K_n}(\Rd)$; the space $\mathrm{L}^\infty(\Rd)$ never enters the topology.
The essential boundedness of the test functions is instead imposed as a \emph{set restriction} on the subclass over which the microlocal
functionals are directly evaluated.
Concretely, the bilinear functional $D$ of Section \ref{sec:main_result} is defined by an explicit integral formula on the dense subspace
$\mathrm{L}^\infty(\Rd)\cap\mathrm{VMO}_c(\Rd)$, where boundedness is genuinely required (see Remark \ref{rem:boundedness}), and is then
extended to all of $\mathrm{VMO}_c(\Rd)$ by continuity, using the fact that the localized $\mathrm{BMO}$ seminorm is a genuine norm on
compactly supported functions (a function of vanishing mean oscillation with support in a proper compact set is constant only if it is zero).

The intersection space $\mathrm{L}^\infty(\Rd)\cap\mathrm{VMO}(\Rd)$ additionally carries a Banach algebra structure that is exploited
when absorbing variable coefficients into the test slot.
By themselves, $\mathrm{BMO}(\Rd)$ and $\mathrm{VMO}(\Rd)$ are linear spaces but fail to be algebras under pointwise multiplication,
as elements can exhibit localized logarithmic singularities.
Intersecting with $\mathrm{L}^\infty(\Rd)$ provides the necessary truncation, establishing a Banach algebra.

Specifically, for any $f, g \in \mathrm{L}^\infty(\Rd) \cap \mathrm{BMO}(\Rd)$, the product $fg$ belongs to the same space, and its
$\mathrm{BMO}$ seminorm satisfies the uniform bound:
\begin{equation*}
\|fg\|_{\mathrm{BMO}(\Rd)} \leq C \Big( \|f\|_{\mathrm{L}^\infty(\Rd)} \|g\|_{\mathrm{BMO}(\Rd)} + \|g\|_{\mathrm{L}^\infty(\Rd)} \|f\|_{\mathrm{BMO}(\Rd)} \Big)
\end{equation*}
where $C>0$ is a dimensional constant.

Because $\mathrm{VMO}(\Rd)$ is a closed subspace, this algebra property extends to $\mathrm{L}^\infty(\Rd) \cap \mathrm{VMO}(\Rd)$.
Furthermore, if one of the functions has compact support, the product inherits it.
Consequently, multiplying a global coefficient $Q \in \mathrm{L}^\infty(\Rd) \cap \mathrm{VMO}(\Rd)$ by a localized test function
$\phi \in \mathrm{L}^\infty(\Rd) \cap \mathrm{VMO}_c(\Rd)$ guarantees that $Q\phi \in \mathrm{L}^\infty(\Rd) \cap \mathrm{VMO}_c(\Rd)$.
This algebraic closure is what allows variable coefficients to be absorbed into the spatial test functions of the corresponding
H-distributions in subsequent sections.

\begin{remark}\label{rem:boundedness}
The essential boundedness of the spatial test functions is not a topological requirement but an \emph{analytic} one, and it enters the
construction in three distinct ways; we stress that these are the analytic roles of the $\mathrm{L}^\infty$ factor, logically separate
from the three structural prerequisites (compact support, the strict integrability gap, and the LF-space topology) collected in Remark \ref{rem:prerequisites}.
First, for $\varphi$ to act as a pointwise multiplier that preserves the underlying Lebesgue spaces, so that $\varphi u_n\in\mathrm{L}^p(\Rd)$
and $\varphi v_n\in\mathrm{L}^q(\Rd)$ and the defining integrals of the H-distribution converge, one needs $\varphi\in\mathrm{L}^\infty(\Rd)$;
a merely $\mathrm{VMO}$ function may carry a local logarithmic singularity and fail to be a bounded multiplier.
Second, the $\mathrm{L}^p$-variant of the First Commutation Lemma (Lemma \ref{1stcommlemma_Lp}), which is precisely what reduces the bilinear
form to a functional of the product $\varphi_1\overline{\varphi_2}$, requires the multiplication symbol to lie in $\mathrm{L}^\infty(\Rd)\cap\mathrm{VMO}(\Rd)$;
the compactness of $[M_b,\mathcal{A}_\psi]$ genuinely fails for unbounded $b$.
Third, the Banach algebra bound stated above, used to absorb the variable coefficients $Q$ and $a_j$ into the test slot, relies on the $\mathrm{L}^\infty$ factor.
Since all three requirements concern only the functions on which $D$ is explicitly evaluated, boundedness is imposed as a set restriction rather than as part of the topology.
The values of the continuous extension of $D$ to unbounded elements of $\mathrm{VMO}_c(\Rd)$, which carry no direct integral interpretation, are never used:
every evaluation appearing in the Localisation Principle and in the applications is of the form $D(Q\phi,1)$ or $D(a_j\varphi,\kappa_j\psi)$ with all factors bounded.
\end{remark}

This intersection space also behaves well under standard regularization.
For any $f \in \mathrm{L}^\infty(\Rd) \cap \mathrm{VMO}(\Rd)$ and a standard smooth mollifier $\eta_\epsilon$, the convolution
$f \ast \eta_\epsilon$ satisfies $\|f \ast \eta_\epsilon\|_{\mathrm{L}^\infty(\Rd)} \leq \|f\|_{\mathrm{L}^\infty(\Rd)}$ and
$\|f \ast \eta_\epsilon\|_{\mathrm{BMO}(\Rd)} \leq C \|f\|_{\mathrm{BMO}(\Rd)}$, with $f \ast \eta_\epsilon \to f$ in the $\mathrm{BMO}$ norm as $\epsilon \to 0$.
This stability under mollification is used when constructing sequence approximations for equations with rough $\mathrm{VMO}$ coefficients.

Furthermore, because $\mathrm{VMO}_c(\Rd)$ is a separable strict LF-space, the space of classical smooth test functions
$\mathcal{D}(\Rd)$ is dense within it in the inductive limit ($\mathrm{BMO}$-based) topology, which bridges the gap between
rough $\mathrm{VMO}$ coefficients and smooth analytical techniques. This topology also shifts the dual space from the global
Hardy space to a Fr\'echet space of local Hardy distributions, which permits the existence of non-trivial, localized microlocal
functionals (Proposition \ref{prop:frechet_dual} in Appendix \ref{sec:appendix_topology}). We emphasize that in the
$\mathrm{L}^p$--$\mathrm{L}^q$ setting the resulting object is a genuine (signed, complex-valued) H-distribution and not a
non-negative measure. This continuous topological embedding is what justifies the application of global Fourier multiplier
operators to sequences within our localized test space.

By $\hat{u}$ or $\mathcal{F}(u)$ we denote the Fourier transform. Throughout, $\kappa$ denotes the least integer strictly greater
than $d/2$, so that $\mathrm{C}^\kappa(\mathrm{S}^{d-1})$ is the Banach space of $\kappa$-times continuously differentiable functions
on the sphere and the H\"ormander--Mikhlin smoothness threshold $\kappa > d/2$ is met. A Fourier multiplier operator with symbol
$\psi \in \mathrm{C}^\kappa(\mathrm{S}^{d-1})$ is denoted by $\mathcal{A}_\psi(u) = \left(\psi(\mxi/|\mxi|) \hat{u}(\mxi)  \right)^\lor$.
We rely on an $\mathrm{L}^p$-variant of the First Commutation Lemma, which uses Krasnoselskij-type interpolation arguments
alongside classical results for the Riesz transform.
A detailed proof of this result can be found in the published article \cite{AMM} as well as in the PhD thesis \cite{mmphd}.

\begin{lemma}\label{1stcommlemma_Lp}
Let $b\in\mathrm{L}^\infty(\Rd)\cap\mathrm{VMO}(\Rd)$ and $\psi\in\mathrm{C}^\kappa(\mathrm{S}^{d-1})$ for $\kappa > d/2$.
Then the commutator $[M_b,\mathcal{A}_\psi] = M_b \mathcal{A}_\psi - \mathcal{A}_\psi M_b$ is a compact operator on $\mathrm{L}^p(\Rd)$
for all $p\in(1,\infty)$.
\end{lemma}

\section{The Main Result: H-Distributions in VMO}\label{sec:main_result}
Because our test space is a strict inductive limit of infinite-dimensional Banach spaces, whose closed bounded sets are non-compact,
it fails the Montel property (Proposition \ref{prop:montel} in Appendix \ref{sec:appendix_topology}). Consequently, bounded sequences
of test functions do not admit strongly convergent subsequences, which is what forces the construction of macroscopic defect measures
to capture the persistent microlocal oscillations. We now state the main theorem establishing the existence of H-distributions with
compactly supported test functions in $\mathrm{VMO}$. We recall that, here and throughout, $\mathrm{VMO}$ denotes the Coifman--Weiss
space (the closure of $\mathrm{C}_c$ in the $\mathrm{BMO}$ norm), not Sarason's space; this is the choice under which the duality
$\mathrm{VMO}(\Rd)'=\mathcal{H}^1(\Rd)$ holds, and hence the choice governing the test space $\mathrm{VMO}_c(\Rd)$ employed below
(see the discussion following \eqref{BMOdef} in Section \ref{sec:notation}).

\begin{theorem}\label{thm:h_dist_vmo}
Let $p, q \in (1, \infty)$ such that $\frac{1}{p} + \frac{1}{q} < 1$.
Let $(u_n)_n$ and $(v_n)_n$ be sequences of functions bounded in $\mathrm{L}^p(\Rd)$ and $\mathrm{L}^q(\Rd)$ respectively.
Assume that $u_n \rightharpoonup 0$ weakly in $\mathrm{L}^p(\Rd)$ and $v_n \rightharpoonup 0$ weakly in $\mathrm{L}^q(\Rd)$.
Then, after passing to a subsequence (still denoted by $u_n, v_n$), there exists a \emph{jointly} continuous bilinear functional $D$
on $\big(\mathrm{L}^\infty(\Rd)\cap\mathrm{VMO}_c(\Rd)\big) \times \mathrm{C}^\kappa(\mathrm{S}^{d-1})$ for $\kappa > d/2$ such that
for all $\varphi_1, \varphi_2 \in \mathrm{L}^\infty(\Rd)\cap\mathrm{VMO}_c(\Rd)$ and $\psi \in \mathrm{C}^\kappa(\mathrm{S}^{d-1})$
it holds:

\begin{equation}\label{eq:h_dist_def}
D(\varphi_1 \overline{\varphi}_2,\psi) = \lim_n \int_{\Rd} \mathcal{A}_\psi(\varphi_1 u_n)(\mx)\overline{(\varphi_2 v_n)(\mx)}\, d\mx = \lim_n \int_{\Rd} \varphi_1(\mx) u_n(\mx) \overline{\left(\mathcal{A}_{\overline{\psi}}(\varphi_2 v_n)\right)(\mx)}\, d\mx.
\end{equation}

Furthermore, for every compact set $K_\varphi \subset \Rd$, there exists a constant $C_{K_\varphi} > 0$ depending on $K_\varphi, p, q, d$,
and the bounds of the sequences, such that for all test functions supported in $K_\varphi$, the functional satisfies:
\begin{equation}\label{eq:h_dist_bound}
|D(\varphi,\psi)|
\leq C_{K_\varphi} \|\varphi\|_{\mathrm{BMO}(\Rd)} \|\psi\|_{\mathrm{C}^\kappa(\mathrm{S}^{d-1})}.
\end{equation}
\end{theorem}

\begin{proof}
Since $u_n \in \mathrm{L}^p(\Rd)$ and $v_n \in \mathrm{L}^q(\Rd)$, and the test functions are essentially bounded with compact support,
the products $\varphi_1 u_n$ and $\varphi_2 v_n$ remain in $\mathrm{L}^p(\Rd)$ and $\mathrm{L}^q(\Rd).$
By the H\"ormander-Mikhlin theorem, since $\psi \in \mathrm{C}^\kappa(\mathrm{S}^{d-1})$, the operator $\mathcal{A}_\psi$ is bounded on
$\mathrm{L}^p(\Rd)$. The adjoint of $\mathcal{A}_\psi$ is $\mathcal{A}_{\overline{\psi}}$, yielding the equivalence of the two integrals.

To show the limit depends only on the product $\varphi = \varphi_1 \overline{\varphi}_2$ and $\psi$, we shift the multiplier to the second
sequence via its adjoint. Because $\varphi_1 u_n \in \mathrm{L}^p(\Rd)$ and $\varphi_2 v_n \in \mathrm{L}^{p'}(\Rd)$ (due to its strict 
compact support and $q > p'$), the standard adjoint relation holds:
\begin{equation*}
\int_{\Rd} \mathcal{A}_\psi(\varphi_1 u_n) \overline{\varphi_2 v_n} \, d\mx = \int_{\Rd} \varphi_1 u_n \overline{\mathcal{A}_{\overline{\psi}}(\varphi_2 v_n)} \, d\mx
\end{equation*}
To isolate $\varphi_2$ from the adjoint multiplier, we invoke the definition of the commutator,
$\mathcal{A}_{\overline{\psi}}(\varphi_2 v_n) = [\mathcal{A}_{\overline{\psi}}, M_{\varphi_2}] v_n + \varphi_2 \mathcal{A}_{\overline{\psi}} v_n$.
Substituting its complex conjugate into the integral yields:
\begin{equation*}
\int_{\Rd} \varphi_1 u_n \overline{\mathcal{A}_{\overline{\psi}}(\varphi_2 v_n)} \, d\mx = \int_{\Rd} \varphi_1 u_n \overline{[\mathcal{A}_{\overline{\psi}}, M_{\varphi_2}] v_n} \, d\mx + \int_{\Rd} \varphi_1 \overline{\varphi}_2 u_n \overline{\mathcal{A}_{\overline{\psi}} v_n} \, d\mx
\end{equation*}

By Lemma \ref{1stcommlemma_Lp}, because $\varphi_2 \in \mathrm{L}^\infty(\Rd) \cap \mathrm{VMO}_c(\Rd)$, the commutator
$[\mathcal{A}_{\overline{\psi}}, M_{\varphi_2}]$ is a compact operator on $\mathrm{L}^q(\Rd)$.
Because $v_n \rightharpoonup 0$ weakly in $\mathrm{L}^q(\Rd)$, the compact operator maps it to a strongly converging sequence:
$[\mathcal{A}_{\overline{\psi}}, M_{\varphi_2}] v_n \to 0$ strongly in $\mathrm{L}^q(\Rd)$.
Because $\varphi_1$ has compact support, it is contained in some compact set $K_{\varphi_1}$.
Since $\frac{1}{p} + \frac{1}{q} < 1$, we apply H\"older's inequality strictly over $K_{\varphi_1}$ to the commutator term:
\begin{equation*}
\left| \int_{K_{\varphi_1}} \varphi_1 u_n \overline{[\mathcal{A}_{\overline{\psi}}, M_{\varphi_2}] v_n} \, d\mx \right| \leq \| \varphi_1 \|_{\mathrm{L}^\infty} \| u_n \|_{\mathrm{L}^p(\Rd)} \left\| [\mathcal{A}_{\overline{\psi}}, M_{\varphi_2}] v_n \right\|_{\mathrm{L}^q(\Rd)} |K_{\varphi_1}|^{1 - \frac{1}{p} - \frac{1}{q}}
\end{equation*}
Because the sequence $[\mathcal{A}_{\overline{\psi}}, M_{\varphi_2}] v_n \to 0$ strongly in $\mathrm{L}^q(\Rd)$ and $u_n$ is
uniformly bounded in $\mathrm{L}^p(\Rd)$, this integral vanishes as $n \to \infty$. 

The limit of the bilinear form reduces entirely to the action on the product $\varphi = \varphi_1 \overline{\varphi}_2$, allowing us to define $D(\varphi, \psi)$. 

To establish the bound on the functional, note that 
$\frac{1}{p} + \frac{1}{q} < 1$ implies there exists a finite exponent $r \in (1, \infty)$ such that $\frac{1}{r} + \frac{1}{p} + \frac{1}{q} = 1$.
The product $\varphi = \varphi_1 \overline{\varphi}_2$ has support restricted to a compact set $K_\varphi$.
Applying the generalized H\"older inequality over this support:
\begin{equation*}
\left| \int_{K_\varphi} \varphi u_n \overline{\mathcal{A}_{\overline{\psi}} v_n} \, d\mx \right| \leq \|\varphi\|_{\mathrm{L}^r(K_\varphi)} \|u_n\|_{\mathrm{L}^p(\Rd)} \|\mathcal{A}_{\overline{\psi}} v_n\|_{\mathrm{L}^q(\Rd)}
\end{equation*}
Using the H\"ormander-Mikhlin bound, $\|\mathcal{A}_{\overline{\psi}} v_n\|_{\mathrm{L}^q} \leq C_{d,q} \|\psi\|_{\mathrm{C}^\kappa} \|v_n\|_{\mathrm{L}^q}$.
It remains to control $\|\varphi\|_{\mathrm{L}^r(K_\varphi)}$ by the $\mathrm{BMO}$ norm. The John-Nirenberg inequality bounds only the oscillation about the mean,
so we first exploit the compact support of $\varphi$ to control the mean itself. Fix a ball $B \supset K_\varphi$ large enough that
$|B \setminus K_\varphi| \geq \tfrac{1}{2}|B|$. Since $\varphi \equiv 0$ on $B \setminus K_\varphi$, the definition of the $\mathrm{BMO}$ seminorm gives
\begin{equation*}
\frac{|B \setminus K_\varphi|}{|B|}\,|\varphi_B| = \frac{1}{|B|}\int_{B \setminus K_\varphi} |\varphi - \varphi_B|\, d\mx \leq \frac{1}{|B|}\int_B |\varphi - \varphi_B|\, d\mx \leq \|\varphi\|_{\mathrm{BMO}(\Rd)},
\end{equation*}
whence $|\varphi_B| \leq 2\,\|\varphi\|_{\mathrm{BMO}(\Rd)}$. Combining this with the John-Nirenberg inequality on $B$, which yields
$\|\varphi - \varphi_B\|_{\mathrm{L}^r(B)} \leq C_{d,r}\,|B|^{1/r}\,\|\varphi\|_{\mathrm{BMO}(\Rd)}$, we obtain
\begin{equation*}
\|\varphi\|_{\mathrm{L}^r(K_\varphi)} \leq \|\varphi - \varphi_B\|_{\mathrm{L}^r(B)} + |\varphi_B|\,|B|^{1/r} \leq C_{\mathrm{BMO}, K_\varphi, r}\,\|\varphi\|_{\mathrm{BMO}(\Rd)},
\end{equation*}
where the constant depends on $r$, the dimension, and the volume of the enclosing ball $B$ (hence on the diameter of $K_\varphi$; see Proposition \ref{prop:JN_constant}).
This step is where the compact support is essential: for a general $\mathrm{BMO}$ function the $\mathrm{L}^r$ norm is not controlled by the seminorm alone, since additive
constants are invisible to the latter.
Combining these yields the uniform bound:
\begin{equation*}
\left| \int_{K_\varphi} \varphi u_n \overline{\mathcal{A}_{\overline{\psi}} v_n} \, d\mx \right|
\leq \left( C_{d,q} C_{\mathrm{BMO}, K_\varphi, r} \sup_n \|u_n\|_{\mathrm{L}^p} \sup_n \|v_n\|_{\mathrm{L}^q} \right) \|\varphi\|_{\mathrm{BMO}(\Rd)} \|\psi\|_{\mathrm{C}^\kappa(\mathrm{S}^{d-1})}
\end{equation*}
By the structural properties of strict LF-spaces, this uniform localized bound is exactly the necessary and sufficient condition to guarantee that $D$ extends to a
jointly continuous bilinear functional on the global product topology (Proposition \ref{prop:joint_cont} in Appendix \ref{sec:appendix_topology}).
Moreover, since this estimate is expressed through the $\mathrm{BMO}$ norm, which is a genuine norm on the Banach step $\mathrm{VMO}_{K_\varphi}(\Rd)$, and since
the bounded functions $\mathrm{L}^\infty(\Rd)\cap\mathrm{VMO}_{K_\varphi}(\Rd)$, on which $D$ was defined by the integral formula, are dense in $\mathrm{VMO}_{K_\varphi}(\Rd)$
in the $\mathrm{BMO}$ topology (Proposition \ref{prop:bounded_density}), the functional $D$ admits a unique jointly continuous extension to all of $\mathrm{VMO}_c(\Rd)$;
the extension carries no separate integral interpretation and is only ever evaluated on bounded functions in the sequel (cf.\ Remark \ref{rem:boundedness}).

To extract a subsequence that converges for all valid test functions, we proceed inductively over an exhausting sequence of compact sets $(K_m)_m$ for $\Rd$,
since the bounding constant $C_{K_m}$ depends on the support diameter.
For a fixed $m$, the established bound guarantees that the sequence of bilinear functionals is equicontinuous on
$\big(\mathrm{L}^\infty(\Rd)\cap\mathrm{VMO}_{K_m}(\Rd)\big) \times \mathrm{C}^\kappa(\mathrm{S}^{d-1})$.
Because the underlying LF-space $\mathrm{VMO}_c(\Rd)$ is separable (Proposition \ref{prop:separable} in Appendix \ref{sec:appendix_topology}), and its bounded subspace
$\mathrm{L}^\infty(\Rd)\cap\mathrm{VMO}_{K_1}(\Rd)$ is therefore separable in the $\mathrm{BMO}$ topology, we can choose dense countable
subsets in $\mathrm{L}^\infty(\Rd)\cap\mathrm{VMO}_{K_1}(\Rd)$ and the symbol space $\mathrm{C}^\kappa(\mathrm{S}^{d-1})$, applying the
Cantor diagonal procedure to extract a subsequence converging on the dense subsets.
By equicontinuity, it converges for all functions supported in $K_1$.
We then extract a sub-subsequence for $K_2$, extending the limit to $\mathrm{VMO}_{K_2}(\Rd)$, and proceed inductively.
The final diagonal subsequence converges for every $\varphi \in \mathrm{L}^\infty(\Rd)\cap\mathrm{VMO}_c(\Rd)$, since any such compactly
supported function eventually belongs to some $K_m$, completing the proof.
\end{proof}

\begin{remark}
We note that the H-distribution $D$ constructed in Theorem \ref{thm:h_dist_vmo} depends on the extracted subsequence $(u_n, v_n)$. As with
classical H-measures, the distribution is not uniquely determined by the original sequences unless those sequences possess a unique,
globally uniform microlocal energy density. Therefore, any physical conclusions or defect representations drawn from $D$ apply only to the
specific concentration and oscillation limits of the chosen subsequence.
\end{remark}

\begin{remark}\label{rem:prerequisites}
The functional-analytic bounds established in this theorem rely on three strict structural prerequisites:
\begin{itemize}
    \item \textbf{Compact Support:} Without restricting the test functions to $\mathrm{VMO}_c(\Rd)$, the $\mathrm{L}^p-\mathrm{L}^q$ product
    would be evaluated over the unbounded domain $\Rd$. Because $\frac{1}{p} + \frac{1}{q} < 1$, the product does not generally belong to
    $\mathrm{L}^1(\Rd)$. The compact support localizes the evaluation to a finite Lebesgue space $\mathrm{L}^r(K_\varphi)$, ensuring the integral is well-defined.
    \item \textbf{Strict Integrability Gap:} In the classical boundary case where $\frac{1}{p} + \frac{1}{q} = 1$, the test function $\varphi$
    must belong to $\mathrm{L}^\infty(\Rd)$ to satisfy integrability. Since the $\mathrm{L}^\infty$-norm cannot be controlled by the $\mathrm{BMO}$-norm,
    continuity is lost. The strict gap $\frac{1}{p} + \frac{1}{q} < 1$ permits the use of the localized John-Nirenberg inequality.
    \item \textbf{The LF-Space Topology:} The bounding constant $C_{K_\varphi}$ generated by the John-Nirenberg inequality grows explicitly
    with the diameter of the support. Thus, $D$ is not uniformly bounded with respect to the global $\mathrm{BMO}(\Rd)$ norm.
    However, because $\mathrm{VMO}_c(\Rd)$ is a strict LF-space, this localized boundedness is the exact necessary and sufficient condition for $D$
    to act as a continuous bilinear functional, mirroring classical Schwartz distributions.
\end{itemize}
\end{remark}

\begin{corollary}\label{cor:Lp_Linfty}
Let $p \in (1, \infty)$. Assume that $(u_n)_n$ is bounded in $\mathrm{L}^p(\Rd)$ and $u_n \rightharpoonup 0$ weakly in $\mathrm{L}^p(\Rd)$.
Assume further that $(v_n)_n$ is a sequence of uniformly compactly supported functions bounded in $\mathrm{L}^\infty(\Rd)$, and $v_n \rightharpoonup^\ast 0$
in the weak-$\ast$ topology of $\mathrm{L}^\infty(\Rd)$.
Then, passing to a subsequence, there exists a continuous bilinear functional $D$ on
$\big(\mathrm{L}^\infty(\Rd)\cap\mathrm{VMO}_c(\Rd)\big) \times \mathrm{C}^\kappa(\mathrm{S}^{d-1})$ satisfying the integral representations and bounds of Theorem \ref{thm:h_dist_vmo}.
\end{corollary}

\begin{proof}
Because $(v_n)_n$ has uniformly compact support and is bounded in $\mathrm{L}^\infty(\Rd)$, it is bounded in $\mathrm{L}^q(\Rd)$ for every finite $q \geq 1$.
For any fixed $p \in (1, \infty)$, we choose $q$ sufficiently large such that $\frac{1}{p} + \frac{1}{q} < 1$.
The existence of the H-distribution follows as a direct application of Theorem \ref{thm:h_dist_vmo}.
\end{proof}

We record three structural properties of the functional $D$ that will be used below and that clarify its nature as a microlocal object.
To display the dependence on the underlying sequences, we write $D_{u,v}$ for the H-distribution associated with the pair $(u_n)\subset\mathrm{L}^p(\Rd)$, $(v_n)\subset\mathrm{L}^q(\Rd)$.

\begin{proposition}[Hermitian symmetry]\label{prop:hermitian}
Let $D_{u,v}$ be as in Theorem \ref{thm:h_dist_vmo}, and let $D_{v,u}$ denote the H-distribution associated, along the same subsequence,
with the swapped pair $(v_n),(u_n)$ (which exists since the integrability condition $\frac1p+\frac1q<1$ is symmetric). Then, for all
$\varphi_1,\varphi_2\in\mathrm{L}^\infty(\Rd)\cap\mathrm{VMO}_c(\Rd)$ and $\psi\in\mathrm{C}^\kappa(\mathrm{S}^{d-1})$,
\begin{equation*}
\overline{D_{u,v}(\varphi_1\overline{\varphi_2},\psi)} = D_{v,u}(\varphi_2\overline{\varphi_1},\overline{\psi}).
\end{equation*}
In particular, in the diagonal case $u_n=v_n$ with $p=q$, the functional is Hermitian: $\overline{D(\varphi_1\overline{\varphi_2},\psi)} = D(\varphi_2\overline{\varphi_1},\overline{\psi})$.
\end{proposition}

\begin{proof}
By the definition of $D_{u,v}$ and conjugation of the integrals,
\begin{equation*}
\overline{D_{u,v}(\varphi_1\overline{\varphi_2},\psi)}
= \overline{\lim_n\int_{\Rd}\mathcal{A}_\psi(\varphi_1 u_n)\overline{\varphi_2 v_n}\, d\mx}
= \lim_n\int_{\Rd}\varphi_2 v_n\,\overline{\mathcal{A}_\psi(\varphi_1 u_n)}\, d\mx.
\end{equation*}
The adjoint of $\mathcal{A}_{\overline{\psi}}$ with respect to the sesquilinear $\mathrm{L}^2$ pairing is $\mathcal{A}_\psi$, so
$\int \mathcal{A}_{\overline{\psi}}(\varphi_2 v_n)\overline{\varphi_1 u_n}\, d\mx = \int \varphi_2 v_n\,\overline{\mathcal{A}_\psi(\varphi_1 u_n)}\, d\mx$;
the pairing is well-defined because $\varphi_2 v_n$ and $\varphi_1 u_n$ are compactly supported and $\frac1p+\frac1q<1$. Hence the right-hand side equals
$\lim_n\int \mathcal{A}_{\overline{\psi}}(\varphi_2 v_n)\overline{\varphi_1 u_n}\, d\mx = D_{v,u}(\varphi_2\overline{\varphi_1},\overline{\psi})$.
\end{proof}

\begin{proposition}[Non-negativity of the diagonal trace]\label{prop:diagonal_positive}
Let $p\in(2,\infty)$ and $q=p$, and let $(w_n)$ be bounded in $\mathrm{L}^p(\Rd)$ with $w_n\rightharpoonup 0$. Let $D$ be the H-distribution of
Theorem \ref{thm:h_dist_vmo} associated with the diagonal pair $u_n=v_n=w_n$. Then the identity-symbol trace $\varphi\mapsto D(\varphi,1)$ is a
non-negative functional: for every $\varphi\in\mathrm{L}^\infty(\Rd)\cap\mathrm{VMO}_c(\Rd)$ with $\varphi\geq 0$,
\begin{equation*}
D(\varphi,1)\geq 0.
\end{equation*}
Moreover $D(\varphi,1)=\int_{\Rd}\varphi\, h\, d\mx$, where $0\leq h\in\mathrm{L}^{p/2}_{loc}(\Rd)$ is the weak-$\mathrm{L}^{p/2}$ limit (along the subsequence)
of $|w_n|^2$. Thus the diagonal trace is represented by a non-negative defect density, recovering the classical H-measure defect within the $\mathrm{VMO}$ framework.
\end{proposition}

\begin{proof}
Since $\psi=1$ gives $\mathcal{A}_1=\mathrm{Id}$, Theorem \ref{thm:h_dist_vmo} yields
\begin{equation*}
D(\varphi,1)=\lim_n\int_{\Rd}\varphi\, w_n\overline{w_n}\, d\mx=\lim_n\int_{\Rd}\varphi\,|w_n|^2\, d\mx.
\end{equation*}
Because $\frac{2}{p}<1$, the sequence $(|w_n|^2)$ is bounded in the reflexive space $\mathrm{L}^{p/2}(\Rd)$. The above limit exists for every
$\varphi\in\mathrm{C}_c^\infty(\Rd)$, so $(|w_n|^2)$ converges weakly in $\mathrm{L}^{p/2}(\Rd)$ to a limit $h$, necessarily $h\geq 0$, with
$D(\varphi,1)=\int_{\Rd}\varphi\, h\, d\mx$. If $\varphi\geq 0$, each integrand $\varphi|w_n|^2$ is non-negative, whence $D(\varphi,1)\geq 0$.
\end{proof}

\begin{proposition}[Spatial locality]\label{prop:spatial_locality}
Let $D$ be as in Theorem \ref{thm:h_dist_vmo}, and let $\Omega\subseteq\Rd$ be open. If $u_n\to 0$ strongly in $\mathrm{L}^p_{loc}(\Omega)$, then
$D(\varphi_1\overline{\varphi_2},\psi)=0$ whenever $\varphi_1$ has compact support in $\Omega$; symmetrically, if $v_n\to 0$ strongly in
$\mathrm{L}^q_{loc}(\Omega)$, then $D(\varphi_1\overline{\varphi_2},\psi)=0$ whenever $\varphi_2$ has compact support in $\Omega$.
Consequently, the spatial support of $D$ is contained in the closed set on which neither sequence converges strongly to zero.
\end{proposition}

\begin{proof}
Suppose $\varphi_1$ has compact support $K\subset\Omega$ and $u_n\to 0$ in $\mathrm{L}^p(K)$. Then $\varphi_1 u_n\to 0$ strongly in $\mathrm{L}^p(\Rd)$,
and by the H\"ormander-Mikhlin bound $\mathcal{A}_\psi(\varphi_1 u_n)\to 0$ strongly in $\mathrm{L}^p(\Rd)$. Since $\varphi_2 v_n$ is bounded in
$\mathrm{L}^q(\Rd)$ with fixed compact support and $q>p'$, it is bounded in $\mathrm{L}^{p'}(\Rd)$, so H\"older's inequality gives
\begin{equation*}
\left|\int_{\Rd}\mathcal{A}_\psi(\varphi_1 u_n)\overline{\varphi_2 v_n}\, d\mx\right|
\leq \|\mathcal{A}_\psi(\varphi_1 u_n)\|_{\mathrm{L}^p(\Rd)}\,\|\varphi_2 v_n\|_{\mathrm{L}^{p'}(\Rd)}\longrightarrow 0.
\end{equation*}
Hence $D(\varphi_1\overline{\varphi_2},\psi)=0$. The case $v_n\to 0$ follows from the adjoint representation of Theorem \ref{thm:h_dist_vmo}.
\end{proof}

\section{The Localisation Principle}\label{sec:localisation}
We now present a localisation principle for H-distributions with test functions in $\mathrm{VMO}_c(\Rd)$.
Because the governing equations are assumed to be in divergence form, the partial derivatives can be shifted onto the momentum test sequence.
This allows the coefficients of the differential operator to reside in $\mathrm{VMO}(\Rd)$ without requiring them to possess weak derivatives.

\begin{theorem}\label{thm:localisation}
Let $p, q \in (1, \infty)$ such that $\frac{1}{p} + \frac{1}{q} < 1$.
Assume that $u_n \rightharpoonup 0$ weakly in $\mathrm{L}^p(\Rd)$ and $v_n \rightharpoonup 0$ weakly in $\mathrm{L}^q(\Rd)$.
Furthermore, assume that $(u_n)$ satisfies the equation:
\begin{equation*}
\sum_{j=1}^d \partial_j(a_j u_n) = f_n
\end{equation*}
where $f_n \to 0$ strongly in $\mathrm{W}^{-1, p}(\Rd)$, and the coefficients satisfy $a_j \in \mathrm{L}^\infty(\Rd) \cap \mathrm{VMO}(\Rd)$.
Then, the corresponding H-distribution $D$ satisfies:
\begin{equation}\label{eq:localisation}
\sum_{j=1}^d D(a_j \varphi, \kappa_j \psi) = 0
\end{equation}
for all $\varphi \in \mathrm{L}^\infty(\Rd) \cap \mathrm{VMO}_c(\Rd)$ and $\psi \in \mathrm{C}^\kappa(\mathrm{S}^{d-1})$, where $\kappa_j(\mxi) = \frac{\xi_j}{|\mxi|}$.
\end{theorem}

\begin{proof}
Let $\varphi \in \mathrm{L}^\infty(\Rd) \cap \mathrm{VMO}_c(\Rd)$ and $\psi \in \mathrm{C}^\kappa(\mathrm{S}^{d-1})$.
To test the differential equation, we introduce a spatial cut-off function $\chi \in \mathrm{C}_c^\infty(\Rd)$ such that $\chi \equiv 1$ on the compact support of $\varphi$.
We construct a testing sequence $w_n$ by applying the Riesz potential of order one:
\begin{equation*}
w_n = \chi \mathcal{A}_{i \overline{\psi}} (-\Delta)^{-1/2} (\varphi v_n)
\end{equation*}
Since $(v_n)$ is bounded in $\mathrm{L}^q(\Rd)$ and $\varphi$ is essentially bounded with compact support, $\varphi v_n$
is a bounded sequence in $\mathrm{L}^q(\Rd)$ with compact support.
The Riesz potential $(-\Delta)^{-1/2}$ maps $\mathrm{L}^q_c(\Rd)$ continuously into $\mathrm{W}^{1, q}_{loc}(\Rd)$,
and the multiplier $\mathcal{A}_{i \overline{\psi}}$ preserves this local regularity.
Because $\frac{1}{p} + \frac{1}{q} < 1$, we have $q > p'$.
The local Sobolev space $\mathrm{W}^{1, q}_{loc}(\Rd)$ embeds continuously into $\mathrm{W}^{1, p'}_{loc}(\Rd)$.
By multiplying with the smooth, compactly supported cut-off function $\chi$, we obtain $w_n \in \mathrm{W}^{1, p'}(\Rd)$.
Thus, $(w_n)$ is uniformly bounded in the dual space of the differential constraint.
Since $f_n \to 0$ strongly in $\mathrm{W}^{-1, p}(\Rd)$, the global dual pairing vanishes in the limit:
\begin{equation*}
\lim_{n \to \infty} \auf f_n, \overline{w_n} \zu = 0
\end{equation*}

Expanding the left-hand side of the differential equation via integration by parts, we shift the derivative onto the testing sequence:
\begin{equation*}
\auf \sum_{j=1}^d \partial_j(a_j u_n), \overline{w_n} \zu = - \sum_{j=1}^d \int_{\Rd} a_j u_n \overline{\partial_j w_n} \, d\mx
\end{equation*}
Applying the product rule, the derivative splits as
$\partial_j w_n = (\partial_j \chi) \mathcal{A}_{i \overline{\psi}} (-\Delta)^{-1/2} (\varphi v_n) + \chi \mathcal{A}_{i \overline{\psi}} \partial_j (-\Delta)^{-1/2} (\varphi v_n)$.
The symbol of $\partial_j (-\Delta)^{-1/2}$ is exactly $\frac{2\pi i \xi_j}{2\pi|\mxi|} = \frac{i \xi_j}{|\mxi|}$.
Thus, the second term simplifies to:
\begin{equation*}
\chi \mathcal{A}_{i \overline{\psi}} \partial_j (-\Delta)^{-1/2} (\varphi v_n) = \chi \mathcal{A}_{-\kappa_j \overline{\psi}} (\varphi v_n)
\end{equation*}
When paired with $a_j u_n$, this term evaluates to:
\begin{equation*}
- \sum_{j=1}^d \int_{\Rd} a_j \chi u_n \overline{\mathcal{A}_{-\kappa_j \overline{\psi}} (\varphi v_n)} \, d\mx = \sum_{j=1}^d \int_{\Rd} (a_j \chi) u_n \overline{\mathcal{A}_{\kappa_j \overline{\psi}} (\varphi v_n)} \, d\mx
\end{equation*}
We stress that global integrability of this pairing is not automatic: under the strict gap $\frac{1}{p} + \frac{1}{q} < 1$
the product $u_n \overline{\mathcal{A}_{\kappa_j\overline\psi}(\varphi v_n)}$ lies only in $\mathrm{L}^s(\Rd)$ with $\frac{1}{s} = \frac{1}{p}+\frac{1}{q} < 1$,
which is not integrable over the unbounded domain $\Rd$. Convergence is secured precisely by the cut-off $\chi$, which localizes the coefficient to
$a_j\chi \in \mathrm{L}^\infty(\Rd)\cap\mathrm{VMO}_c(\Rd)$ before the non-local operator $\mathcal{A}_{\kappa_j\overline\psi}$ is applied,
restricting the effective domain of integration to $\mathrm{supp}\,\chi$.
Taking the limit as $n \to \infty$, this yields $D(a_j \chi \overline{\varphi}, \kappa_j \psi)$.
Since $\chi \equiv 1$ on the support of $\varphi$, we have $\chi \overline{\varphi} = \overline{\varphi}$, reducing this term to $D(a_j \overline{\varphi}, \kappa_j \psi)$.
It remains to show that the integral corresponding to the first term, $(\partial_j \chi) \mathcal{A}_{i \overline{\psi}} (-\Delta)^{-1/2} (\varphi v_n)$, vanishes.
Let $S(\varphi v_n) = \mathcal{A}_{i \overline{\psi}} (-\Delta)^{-1/2} (\varphi v_n)$. This sequence is bounded and weakly null in $\mathrm{W}^{1, q}_{loc}(\Rd)$.
Because $\frac{1}{p} + \frac{1}{q} < 1$, we have $q > p'$, which ensures that the local Rellich-Kondrachov embedding
$\mathrm{W}^{1, q}(K) \hookrightarrow\hookrightarrow \mathrm{L}^{p'}(K)$ is compact on any bounded domain $K$.
Since the multiplier $(\partial_j \chi)$ is strictly supported on a compact set, the pointwise product $(\partial_j \chi) S(\varphi v_n)$
converges strongly to zero in $\mathrm{L}^{p'}(\Rd)$.
Paired with the $\mathrm{L}^p$-bounded sequence $a_j u_n$, this term vanishes in the limit.
Since $\mathrm{VMO}(\Rd)$ is stable under complex conjugation, we can replace $\overline{\varphi}$ with $\varphi$, completing the proof.
\end{proof}

\section{The Macroscopic Defect Representation}\label{sec:defect_rep}
In physical applications, macroscopic observables such as energy dissipation, transport rates, and acoustic scattering are mathematically
represented by the weak limit of variable bilinear forms. To evaluate these observables without relying on unphysical uniform continuity
assumptions, we must link the macroscopic limit of the sequence to the microlocal H-distribution.
Let $Q \in \mathrm{L}^\infty(\Rd) \cap \mathrm{VMO}(\Rd)$ be a variable material coefficient, and define the sesquilinear form $q(\mx; \lambda, \eta) := Q(\mx)\lambda\overline{\eta}$.
The following theorem shows that the macroscopic defect of this sesquilinear form is quantified by the trace of the H-distribution.
Because the H-distribution is jointly continuous on the product topology, the projective tensor product factorization of
Proposition \ref{prop:operator_rep} shows that it acts as a continuous linear operator from the spatial test space into the space of
spherical distributions, mapping the geometric location of trapped energy to its directional scattering signature;
Corollary \ref{cor:angular_defect} below makes this operator-theoretic reading explicit and identifies the scalar defect of
Theorem \ref{thm:defect_rep} as its constant-symbol component.

\begin{theorem}\label{thm:defect_rep}
Let $p, q \in (1, \infty)$ such that $\frac{1}{p} + \frac{1}{q} < 1$. Assume that $(u_n)$ is bounded in $\mathrm{L}^p(\Rd)$ and
$u_n \rightharpoonup u$ weakly. Assume that $(v_n)$ is bounded in $\mathrm{L}^q(\Rd)$ and $v_n \rightharpoonup v$ weakly.
If $q(\mx; u_n, v_n) \to \omega$ in the sense of distributions $\mathcal{D}'(\Rd)$, and $D$ is the H-distribution associated
with the purely oscillatory components $(u_n - u)$ and $(v_n - v)$, then the macroscopic defect is exactly quantified by the
H-distribution acting on any test function $\phi \in \mathrm{C}_c^\infty(\Rd)$ as:
\begin{equation}\label{eq:defect_rep}
\auf \omega - q(\mx; u, v), \phi \zu = D(Q\phi, 1).
\end{equation}
\end{theorem}

\begin{proof}
Let $\phi \in \mathrm{L}^\infty(\Rd) \cap \mathrm{VMO}_c(\Rd)$ be a test function. Because $Q \in \mathrm{L}^\infty(\Rd) \cap \mathrm{VMO}(\Rd)$,
the product $Q\phi$ inherits both compact support and $\mathrm{VMO}$ regularity, acting as a valid spatial multiplier.
We evaluate the distributional limit of the sesquilinear form:
\begin{equation*}
\omega(\phi) = \lim_{n \to \infty} \int_{\Rd} Q(\mx) \phi(\mx) u_n \overline{v_n} \, d\mx.
\end{equation*}
We decompose the sequence algebraically by centering around the weak limits:
\begin{align*}
\int_{\Rd} Q\phi u_n \overline{v_n} \, d\mx &= \int_{\Rd} Q\phi (u_n - u + u)\overline{(v_n - v + v)} \, d\mx \\
&= \int_{\Rd} Q\phi (u_n - u)\overline{(v_n - v)} \, d\mx + \int_{\Rd} Q\phi (u_n - u)\overline{v} \, d\mx \\
&\quad + \int_{\Rd} Q\phi u \overline{(v_n - v)} \, d\mx + \int_{\Rd} Q\phi u \overline{v} \, d\mx.
\end{align*}

Because the spatial test function $\phi \in \mathrm{VMO}_c(\Rd)$ inherently possesses compact support, the evaluation domain is
restricted to $K_\phi = \mathrm{supp}\,\phi$. The strict integrability gap $\frac{1}{p} + \frac{1}{q} < 1$ implies $q > p'$,
which guarantees the continuous local Lebesgue embedding $\mathrm{L}^q(K_\phi) \hookrightarrow \mathrm{L}^{p'}(K_\phi)$.
Consequently, the fixed multiplier functions $Q\phi \overline{v}$ and $Q\phi u$ belong to the dual spaces $\mathrm{L}^{p'}(\Rd)$
and $\mathrm{L}^{q'}(\Rd)$, respectively. Thus, the mixed cross-terms represent valid dual pairings against weakly null sequences
and vanish entirely in the limit:
\begin{equation*}
\lim_{n \to \infty} \int_{\Rd} Q\phi (u_n - u)\overline{v} \, d\mx = 0, \quad \lim_{n \to \infty} \int_{\Rd} Q\phi u \overline{(v_n - v)} \, d\mx = 0.
\end{equation*}

Rearranging the surviving terms isolates the definition of the H-distribution evaluated at the identity momentum symbol $\psi = 1$:
\begin{equation*}
\omega(\phi) - \int_{\Rd} q(\mx; u, v) \phi \, d\mx = \lim_{n \to \infty} \int_{\Rd} Q\phi (u_n - u)\overline{(v_n - v)} \, d\mx = D(Q\phi, 1).
\end{equation*}
Since this holds for any valid test function, the defect representation is established in $\mathcal{D}'(\Rd)$.
\end{proof}

The defect representation of Theorem \ref{thm:defect_rep} evaluates the H-distribution only at the constant momentum symbol $\psi = 1$.
The operator representation of Proposition \ref{prop:operator_rep} shows that this scalar defect is the constant-symbol component of a
richer object: the full angular defect distribution carried by the entire symbol slot. We record this explicitly, as it is the point at
which the projective tensor product structure of Appendix \ref{sec:appendix_topology} is used rather than merely descriptive.

\begin{corollary}[Angular defect distribution]\label{cor:angular_defect}
Under the hypotheses of Theorem \ref{thm:defect_rep}, fix $Q \in \mathrm{L}^\infty(\Rd) \cap \mathrm{VMO}(\Rd)$, and equip the Banach algebra
$\mathrm{L}^\infty(\Rd) \cap \mathrm{VMO}_c(\Rd)$ with its \emph{intersection topology}, that is, the strict inductive limit of the Banach-algebra
norms $\|\cdot\|_{\mathrm{L}^\infty(\Rd)} + \|\cdot\|_{\mathrm{BMO}(\Rd)}$ on the steps $\mathrm{L}^\infty(\Rd) \cap \mathrm{VMO}_{K_n}(\Rd)$. Then the map
\begin{equation*}
\phi \longmapsto T_{Q\phi} := D(Q\phi, \cdot)
\end{equation*}
is a continuous linear operator from $\mathrm{L}^\infty(\Rd) \cap \mathrm{VMO}_c(\Rd)$ into the space of spherical distributions
$\big(\mathrm{C}^\kappa(\mathrm{S}^{d-1})\big)'$ of order $\kappa$, and the scalar macroscopic defect \eqref{eq:defect_rep} of
Theorem \ref{thm:defect_rep} is recovered as the pairing of $T_{Q\phi}$ against the constant symbol $\mathbf{1} \in \mathrm{C}^\kappa(\mathrm{S}^{d-1})$:
\begin{equation}\label{eq:angular}
\auf \omega - q(\mx; u, v), \phi \zu = \auf T_{Q\phi}, \mathbf{1} \zu = D(Q\phi, 1).
\end{equation}
Thus $T_{Q\phi}$ resolves the directional (angular) structure of the macroscopic defect at the spatial location selected by $\phi$,
of which the scalar defect $D(Q\phi,1)$ is the zeroth angular moment.
\end{corollary}

\begin{proof}
We first show that the multiplication operator $M_Q\colon \phi \mapsto Q\phi$ is continuous on $\mathrm{L}^\infty(\Rd) \cap \mathrm{VMO}_c(\Rd)$
for the intersection topology. Fix a compact set $K$ and let $\phi \in \mathrm{L}^\infty(\Rd) \cap \mathrm{VMO}_K(\Rd)$. Since
$Q \in \mathrm{L}^\infty(\Rd) \cap \mathrm{VMO}(\Rd)$ and $\mathrm{supp}\,\phi \subseteq K$, the product $Q\phi$ is again supported in $K$, and the
Banach algebra estimates of Section \ref{sec:notation} give
\begin{equation*}
\begin{aligned}
\|Q\phi\|_{\mathrm{L}^\infty(\Rd)} &\le \|Q\|_{\mathrm{L}^\infty(\Rd)}\,\|\phi\|_{\mathrm{L}^\infty(\Rd)}, \\
\|Q\phi\|_{\mathrm{BMO}(\Rd)} &\le C\big(\|Q\|_{\mathrm{L}^\infty(\Rd)}\|\phi\|_{\mathrm{BMO}(\Rd)} + \|Q\|_{\mathrm{BMO}(\Rd)}\|\phi\|_{\mathrm{L}^\infty(\Rd)}\big).
\end{aligned}
\end{equation*}
Writing $\|\cdot\|_\cap := \|\cdot\|_{\mathrm{L}^\infty(\Rd)} + \|\cdot\|_{\mathrm{BMO}(\Rd)}$ for the Banach-algebra norm and adding these two bounds yields
\begin{equation*}
\|Q\phi\|_\cap \le (1 + C)\,\|Q\|_\cap\,\|\phi\|_\cap .
\end{equation*}
Thus $M_Q$ maps the Banach step $\mathrm{L}^\infty(\Rd) \cap \mathrm{VMO}_K(\Rd)$ boundedly into itself; as $K$ was arbitrary, the universal
property of the strict inductive limit shows that $M_Q$ is a continuous linear operator on $\mathrm{L}^\infty(\Rd) \cap \mathrm{VMO}_c(\Rd)$ for the
intersection topology.

By Theorem \ref{thm:h_dist_vmo} the partial evaluation operator $\varphi \mapsto D(\varphi, \cdot)$ of Proposition \ref{prop:operator_rep} is a
continuous linear map from $\mathrm{VMO}_c(\Rd)$ into $\big(\mathrm{C}^\kappa(\mathrm{S}^{d-1})\big)'$, since the bound \eqref{eq:h_dist_bound}
controls it through the $\mathrm{BMO}$ norm alone; a fortiori it is continuous for the finer intersection topology, under which
$\|\cdot\|_{\mathrm{BMO}(\Rd)} \le \|\cdot\|_\cap$. Composing $M_Q$ with this partial evaluation shows that $\phi \mapsto T_{Q\phi}$ is a continuous
linear operator into the spherical distributions of order $\kappa$. Evaluating the spherical distribution $T_{Q\phi}$ against the constant symbol
$\mathbf{1}$ gives $\auf T_{Q\phi}, \mathbf{1} \zu = D(Q\phi, 1)$, which by Theorem \ref{thm:defect_rep} equals $\auf \omega - q(\mx; u, v), \phi \zu$.
\end{proof}

\begin{remark}\label{rem:intersection}
The $\mathrm{L}^\infty$ factor carried by the intersection norm in Corollary \ref{cor:angular_defect} is genuinely needed, and cannot be dropped
in favour of the coarser $\mathrm{BMO}$ inductive-limit topology of the test space. Continuity of $M_Q$ for the latter would amount to $M_Q$ being
bounded on $\mathrm{BMO}$, that is, to $Q$ being a pointwise multiplier of $\mathrm{BMO}$, which is strictly stronger than requiring
$Q \in \mathrm{L}^\infty(\Rd) \cap \mathrm{VMO}(\Rd)$. The obstruction is that the second term of the algebra bound carries the factor $\|\phi\|_{\mathrm{L}^\infty(\Rd)}$,
which is not controlled by $\|\phi\|_{\mathrm{BMO}(\Rd)}$ alone. This is one more instance of the boundedness hypothesis playing a genuine analytic
role, here the absorption of the variable coefficient $Q$ into the test slot, exactly as recorded in Remark \ref{rem:boundedness}.
\end{remark}

The Localisation Principle also furnishes a $\mathrm{VMO}$ counterpart of the classical compensated compactness mechanism of Murat and
Tartar \cite{Murat, Tartar79}: when the constraint forces the localized trace of the H-distribution to vanish, the bilinear
product passes to the limit despite the absence of strong convergence. Within the $\mathrm{VMO}$ framework this mechanism is entirely
geometric, the constraint acting solely through the Localisation Principle.

\begin{corollary}\label{cor:comp_comp}
Assume the sequences $(u_n)$ and $(v_n)$ from Theorem \ref{thm:defect_rep} satisfy the single divergence-form constraint
\begin{equation*}
\sum_{j=1}^d \partial_j(a_j u_n) = f_n \to 0 \quad \text{in } \mathrm{W}^{-1,p}(\Rd), \qquad a_j \in \mathrm{L}^\infty(\Rd) \cap \mathrm{VMO}(\Rd).
\end{equation*}
Then, by the Localisation Principle (Theorem \ref{thm:localisation}), applied to the centred, weakly null sequence $(u_n - u)$ with which the
H-distribution is associated (which satisfies $\sum_{j=1}^d \partial_j(a_j(u_n - u)) = f_n \to 0$ in $\mathrm{W}^{-1, p}(\Rd)$, since
$a_j u_n \rightharpoonup a_j u$ forces $\sum_{j=1}^d \partial_j(a_j u) = 0$ by uniqueness of weak limits), the corresponding H-distribution satisfies
\begin{equation*}
\sum_{j=1}^d D(a_j \varphi, \kappa_j \psi) = 0 \qquad \text{for all } \varphi \in \mathrm{L}^\infty(\Rd) \cap \mathrm{VMO}_c(\Rd),\ \psi \in \mathrm{C}^\kappa(\mathrm{S}^{d-1}),
\end{equation*}
which confines the defect to the characteristic variety $\sum_{j=1}^d a_j(\mx)\xi_j = 0$. If this identity forces the localized trace
to vanish, $D(Q\phi, 1) = 0$ for all $\phi \in \mathrm{C}_c^\infty(\Rd)$, then the sequence exhibits compensated compactness, and the
macroscopic product passes to the limit exactly:
\begin{equation*}
\lim_{n \to \infty} q(\mx; u_n, v_n) = q(\mx; u, v) \quad \text{in } \mathcal{D}'(\Rd).
\end{equation*}
\end{corollary}

\section{Canonical Choices for Testing Sequences}\label{sec:canonical}
In the classical $\mathrm{L}^p-\mathrm{L}^{p'}$ framework of H-distributions, given a weakly converging sequence $u_n \in \mathrm{L}^p(\Rd)$,
a standard choice for the testing sequence is generated pointwise via the Nemytskii operator: $v_n = f(u_n)$.
Because the exponents are conjugate, standard growth bounds imply $v_n \in \mathrm{L}^{p'}(\Rd)$.
The introduction of $\mathrm{VMO}$ test functions alters the admissible test spaces.
Because our construction requires $\frac{1}{p} + \frac{1}{q} < 1$, standard Nemytskii operators may map $v_n$ outside the target space $\mathrm{L}^q(\Rd)$.
By Corollary \ref{cor:Lp_Linfty}, we can define three alternatives for generating a testing sequence $v_n$ related to the primary sequence $u_n$.

\subsection{The Localized Sub-Critical Nemytskii Operator}\label{subsec:subcrit}
We can construct a testing sequence $v_n$ pointwise from $u_n$ via a continuous function $f: \mathbb{R} \to \mathbb{R}$.
However, because the sequences reside in Lebesgue spaces over the unbounded domain $\Rd$, the lack of global decay at infinity restricts standard mappings.
To map $\mathrm{L}^p(\Rd)$ into $\mathrm{L}^q(\Rd)$ globally over this infinite-measure domain, the classical superposition (Nemytskii) theorem requires,
for continuous $f$ with $f(0)=0$, the two-sided scaling bound $|f(s)| \leq C|s|^{p/q}$ to hold for \emph{all} $s \in \mathbb{R}$: the exponent $p/q$
simultaneously caps the growth as $|s| \to \infty$ and forces decay at least as fast as $|s|^{p/q}$ as $s \to 0$, the latter being what controls
integrability on the tails of $\Rd$ where $u_n$ is small.
In many physical applications, the governing non-linearities exhibit strictly sub-critical polynomial growth $|f(s)| \leq C|s|^\gamma$ where $\gamma \leq p/q$.
Because sub-critical functions decay too slowly for $|s| < 1$, they fail to preserve global integrability on the tails of $\Rd$.
To use sub-critical operators within the H-distribution framework, we introduce a spatial cut-off $\chi \in \mathrm{C}_c^\infty(\Rd)$.
Let $K = \mathrm{supp}\,\chi$. Because $K$ is bounded, the local Lebesgue embedding $\mathrm{L}^p(K) \hookrightarrow \mathrm{L}^{\gamma q}(K)$ holds continuously.
We define the testing sequence:
\begin{equation*}
v_n = \chi(\mx) f(u_n)
\end{equation*}
The cut-off ensures that $v_n$ belongs to $\mathrm{L}^q_c(\Rd)$, which fits the $\mathrm{L}^p-\mathrm{L}^q$ framework without requiring artificial critical-growth bounds.

\subsection{The Localized Bounded Transformation}
By Corollary \ref{cor:Lp_Linfty}, we can bypass algebraic growth restrictions by pairing a bounded nonlinearity with a spatial cut-off.
This approach is particularly suitable for measuring the local phase behavior of highly oscillatory sequences in fluid dynamics and transport theory.
Let $\chi \in \mathrm{C}_c^\infty(\Rd)$ be a fixed spatial cut-off function, and let $g: \mathbb{R} \to \mathbb{R}$ be a globally bounded continuous function.
We define the localized testing sequence:
\begin{equation*}
v_n = \chi(\mx) g(u_n)
\end{equation*}

Because $g$ is bounded, $v_n$ is uniformly bounded in $\mathrm{L}^\infty(\Rd)$.
Because of $\chi$, it has uniformly compact support and satisfies the prerequisites of the Macroscopic Defect Representation (Theorem \ref{thm:defect_rep}).
By the DiPerna-Lions renormalization theory, composing a sequence with a globally bounded transformation preserves its characteristic geometry,
allowing this choice to translate linear PDE constraints directly into the H-distribution framework.

\subsection{The Zero-Order Phase-Shift Sequence}
Because the H-distribution framework natively evaluates non-local pseudo-differential operators, the testing sequence $v_n$ can be constructed via a zero-order Fourier multiplier.
Unlike fractional integration, zero-order multipliers do not confer a gain in fractional Sobolev regularity, ensuring the generated sequence
does not inadvertently cross the threshold into local strong compactness.
Assuming the primary sequence is bounded in $\mathrm{L}^p(\Rd)$ for $p > 2$, we define a symmetric integrability pairing $q = p$,
which natively satisfies the strict integrability gap $\frac{1}{p} + \frac{1}{q} = \frac{2}{p} < 1$.
Let $m \in \mathrm{C}^\kappa(\mathrm{S}^{d-1})$ be a zero-order symbol. Applying a spatial cut-off $\chi \in \mathrm{C}_c^\infty(\Rd)$,
we construct the sequence:
\begin{equation*}
v_n = \chi \cdot \mathcal{A}_m(u_n)
\end{equation*}

By the classical Calder\'on-Zygmund theorem, $\mathcal{A}_m$ acts continuously on $\mathrm{L}^p(\Rd)$.
Therefore, $v_n$ is bounded in $\mathrm{L}^p_c(\Rd)$, which matches the required target space.
This non-local choice is applicable when analyzing the cross-energy, wave scattering, or vorticity interactions between a highly
oscillatory field and its orthogonal harmonic projection.

\begin{remark}
In the canonical constructions above, the introduction of the artificial spatial cut-off $\chi \in \mathrm{C}_c^\infty(\Rd)$ is
a mathematical necessity to satisfy the global Lebesgue integrability requirements. However, the resulting macroscopic defect
representation is physically consistent and strictly independent of the specific choice of $\chi$. Provided that $\chi \equiv 1$
on the compact support of the macroscopic test function $\phi$, the H-distribution evaluates only the local microlocal oscillations.
The artificial algebraic tails of $\chi$ are paired against zero within the dual evaluation, rendering the cut-off entirely invisible
to the physical defect measure.
\end{remark}

\subsection{A Unified Macroscopic Defect Theorem}\label{subsec:unified}
The three canonical constructions above, and the three physical applications developed in the sections that follow, share a single
logical skeleton: a testing sequence is produced from $(u_n)$ by an admissible operator, shown to land in $\mathrm{L}^q_c(\Rd)$,
and then fed into the Macroscopic Defect Representation and the Localisation Principle. We isolate this common core once, so that
each subsequent application reduces to the verification specific to its testing operator; the individual physical settings are
then presented as worked examples that retain their modelling content while invoking the abstract result below.

\begin{definition}\label{def:admissible}
Let $q \in (1, \infty)$. A sequence $(v_n)$ is an \emph{admissible testing sequence} in $\mathrm{L}^q_c(\Rd)$ if there exists a
single compact set $K \subset \Rd$ with $\mathrm{supp}\,v_n \subseteq K$ for all $n$, the sequence is uniformly bounded in
$\mathrm{L}^q(\Rd)$, and $v_n \rightharpoonup v$ weakly in $\mathrm{L}^q(\Rd)$.
\end{definition}

\begin{theorem}\label{thm:unified_defect}
Let $p, q \in (1, \infty)$ such that $\frac{1}{p} + \frac{1}{q} < 1$.
Assume that $(u_n)$ is bounded in $\mathrm{L}^p(\Rd)$ and $u_n \rightharpoonup u$ weakly, and satisfies the divergence-form constraint
\begin{equation*}
\sum_{j=1}^d \partial_j(a_j u_n) = f_n \to 0 \quad \text{in } \mathrm{W}^{-1, p}(\Rd), \qquad a_j \in \mathrm{L}^\infty(\Rd) \cap \mathrm{VMO}(\Rd).
\end{equation*}
Let $(v_n)$ be any admissible testing sequence in $\mathrm{L}^q_c(\Rd)$ with weak limit $v$, and let $D$ be the H-distribution
associated, along a subsequence, with the oscillatory components $(u_n - u)$ and $(v_n - v)$. Then for every coefficient
$Q \in \mathrm{L}^\infty(\Rd) \cap \mathrm{VMO}(\Rd)$ such that $Q(\mx) u_n \overline{v_n} \to \omega$ in the sense of distributions $\mathcal{D}'(\Rd)$:
\begin{itemize}
\item[(i)] the macroscopic defect is exactly quantified by the H-distribution, acting on any test function $\phi \in \mathrm{C}_c^\infty(\Rd)$ as
\begin{equation*}
\auf \omega - Q(\mx) u \overline{v}, \phi \zu = D(Q\phi, 1);
\end{equation*}
\item[(ii)] the defect is geometrically trapped on the characteristic variety $\sum_{j=1}^d a_j(\mx)\xi_j = 0$, in the sense that
\begin{equation*}
\sum_{j=1}^d D(a_j \varphi, \kappa_j \psi) = 0
\end{equation*}
for all $\varphi \in \mathrm{L}^\infty(\Rd) \cap \mathrm{VMO}_c(\Rd)$ and $\psi \in \mathrm{C}^\kappa(\mathrm{S}^{d-1})$.
\end{itemize}
\end{theorem}

\begin{proof}
Because $(v_n)$ is an admissible testing sequence, it is uniformly bounded in $\mathrm{L}^q(\Rd)$ with support in a fixed
compact set and $v_n \rightharpoonup v$; together with the strict gap $\frac{1}{p} + \frac{1}{q} < 1$, the pair $(u_n), (v_n)$
satisfies the hypotheses of the Macroscopic Defect Representation (Theorem \ref{thm:defect_rep}), which yields (i).
Because $(u_n)$ satisfies the divergence-form constraint with coefficients $a_j \in \mathrm{L}^\infty(\Rd) \cap \mathrm{VMO}(\Rd)$, it remains
to apply the Localisation Principle (Theorem \ref{thm:localisation}) to the centred, weakly null sequence $(u_n - u)$ with which $D$ is
associated; the Localisation Principle requires this sequence to satisfy the same divergence-form constraint with vanishing right-hand side.
This holds: since each $a_j \in \mathrm{L}^\infty(\Rd)$ is fixed, $u_n \rightharpoonup u$ in $\mathrm{L}^p(\Rd)$ implies $a_j u_n \rightharpoonup a_j u$
in $\mathrm{L}^p(\Rd)$, whence $\sum_{j=1}^d \partial_j(a_j u_n) \rightharpoonup \sum_{j=1}^d \partial_j(a_j u)$ in $\mathrm{W}^{-1, p}(\Rd)$; as the
left-hand side equals $f_n \to 0$ strongly, and hence weakly, uniqueness of weak limits forces $\sum_{j=1}^d \partial_j(a_j u) = 0$, and subtracting
gives $\sum_{j=1}^d \partial_j(a_j(u_n - u)) = f_n \to 0$ strongly in $\mathrm{W}^{-1, p}(\Rd)$. The Localisation Principle thus applies to
$(u_n - u)$ and yields (ii).
\end{proof}

\section{Application: Stratified Transport in VMO Environments}\label{sec:transport}
To illustrate the scalar $\mathrm{VMO}$ H-distribution framework in a different setting, we consider stationary advection
and reaction in highly heterogeneous porous media or stratified fluids.

\subsection{Physical Motivation and the VMO Topology}
In geophysical fluid dynamics, the transport of a scalar quantity, such as a chemical pollutant, heat density, or fluid saturation,
through a sedimentary rock formation is governed by a first-order transport equation.
If the material is highly stratified, such as sandstone layered with shale, the background velocity field
$\mathbf{a}(\mx) = (a_1(\mx), \dots, a_d(\mx))$ varies strongly across the distinct geological strata.
Classical transport theory relies on Lipschitz, or at least uniformly continuous ($\mathrm{C}^0$), velocity fields to
define stable characteristic curves and extract uniform bounds (cf.\ \cite{DiPernaLions}).
While the geological interfaces between distinct rock strata do possess a microscopic transition zone where physical
properties change continuously, this zone is very thin.
Consequently, the classical modulus of continuity for the velocity field is poor.
Any mathematical estimates or compactness arguments relying on classical uniform continuity blow up as the transition
thickness approaches zero.
This is where the functional space $\mathrm{L}^\infty(\Rd; \Rd) \cap \mathrm{VMO}(\Rd; \Rd)$ becomes a natural modeling
choice; indeed, $\mathrm{VMO}$ is the natural regularity threshold at which second-order elliptic problems with rough
coefficients, in both nondivergence and divergence form, retain their $\mathrm{L}^p$ estimates \cite{CFF, DiFazio}.
The $\mathrm{L}^\infty$ restriction enforces the physical reality that the fluid velocity cannot diverge to infinity.
The $\mathrm{VMO}$ regularity accommodates the steepness of the boundary transitions.
Within a single homogeneous stratum, the velocity is steady, yielding zero mean oscillation.
Across the thin geological boundary layers, the velocity transitions rapidly.
However, because $\mathrm{VMO}$-based estimates (such as the John-Nirenberg bounds used in our H-distribution construction)
depend only on the $\mathrm{BMO}$ seminorm and its vanishing modulus, they are independent of the gradient or the classical
modulus of continuity of the coefficient.
This allows the extraction of uniform limits even when the macroscopic velocity field exhibits sharp, near-discontinuous
transitions between strata, while remaining, as a $\mathrm{VMO}$ field, of vanishing mean oscillation and hence free of genuine jumps.
Furthermore, $\mathrm{VMO}$ excludes chaotic, fractal ``white noise'' heterogeneities at all microscopic scales, reflecting
that geological formations, no matter how stratified, are composed of distinct, locally ordered domains.
In this framework, let the scalar sequence $(u_n)$ represent highly oscillatory approximations of the pollutant concentration.
The conservation of mass dictates the linear differential constraint $\sum_{j=1}^d \partial_j(a_j u_n) = f_n$, where $f_n \to 0$ strongly in $\mathrm{W}^{-1,p}(\Rd)$.
Suppose the pollutant undergoes a chemical reaction or is absorbed by the porous medium.
To prevent unphysical infinite reaction rates, the induced reaction is modeled by a globally bounded, continuous, monotonically increasing scalar function $g(u_n)$.
The distributional limit $\omega$ of the scalar product $u_n g(u_n)$ thus represents the macroscopic,
observable reaction rate, ready to be analyzed to determine the exact geometric support of its defect.

\subsection{Characteristic Support of the Transport Defect}
Rather than relying on algebraic monotonicity, we use the H-distribution framework to determine the geometric structure of the macroscopic reaction defect.
By the classical DiPerna-Lions renormalization theory \cite{DiPernaLions}, if the primary sequence $(u_n)$ satisfies the transport equation,
the bounded composition $v_n = g(u_n)$ satisfies the identical differential constraint.

\begin{theorem}\label{thm:transport_app}
Let $p, q \in (1, \infty)$ such that $\frac{1}{p} + \frac{1}{q} < 1$.
Assume $(u_n)$ is bounded in $\mathrm{L}^p(\Rd)$ and $u_n \rightharpoonup u$ weakly.
Let $g$ be a globally bounded, continuous function, and define the localized testing sequence $v_n = \chi(\mx) g(u_n)$
for a non-negative cut-off $\chi \in \mathrm{C}_c^\infty(\Rd)$, with $v_n \rightharpoonup v$ weakly in $\mathrm{L}^q(\Rd)$.

Assume the primary sequence satisfies the transport constraint:
\begin{equation*}
\sum_{j=1}^d \partial_j(a_j u_n) = f_n \to 0 \quad \text{in } \mathrm{W}^{-1, p}(\Rd)
\end{equation*}
where the velocity components satisfy $a_j \in \mathrm{L}^\infty(\Rd) \cap \mathrm{VMO}(\Rd)$.

For any coefficient $Q \in \mathrm{L}^\infty(\Rd) \cap \mathrm{VMO}(\Rd)$, if the macroscopic reaction sequence
$Q(\mx) u_n v_n \to \omega$ in the sense of distributions $\mathcal{D}'(\Rd)$, then the macroscopic defect is exactly
quantified by the H-distribution acting on any test function $\phi \in \mathrm{C}_c^\infty(\Rd)$ as:
\begin{equation*}
\auf \omega - Q(\mx) u v, \phi \zu = D(Q\phi, 1).
\end{equation*}
Furthermore, this defect measure is geometrically trapped, strictly supported on the characteristic variety of the fluid flow:
\begin{equation*}
\sum_{j=1}^d D(a_j \varphi, \kappa_j \psi) = 0
\end{equation*}
for all $\varphi \in \mathrm{L}^\infty(\Rd) \cap \mathrm{VMO}_c(\Rd)$ and $\psi \in \mathrm{C}^\kappa(\mathrm{S}^{d-1})$.
\end{theorem}

\begin{proof}
It suffices to verify that $(v_n)$ is an admissible testing sequence in $\mathrm{L}^q_c(\Rd)$ (Definition \ref{def:admissible});
both conclusions then follow at once from the Unified Macroscopic Defect Theorem (Theorem \ref{thm:unified_defect}).
Since $g$ is globally bounded, $v_n = \chi g(u_n)$ is uniformly bounded in $\mathrm{L}^\infty(\Rd)$ with support contained
in the fixed compact set $\mathrm{supp}\,\chi$, hence uniformly bounded in $\mathrm{L}^q(\Rd)$; together with the assumed
weak convergence $v_n \rightharpoonup v$ this exhibits $(v_n)$ as admissible. Theorem \ref{thm:unified_defect} then delivers
both the exact reaction-defect representation $\auf \omega - Q(\mx) u v, \phi \zu = D(Q\phi, 1)$ and the characteristic
confinement $\sum_{j=1}^d D(a_j \varphi, \kappa_j \psi) = 0$, so the defect measure is entirely supported on the
characteristic variety $\sum_{j=1}^d a_j(\mx) \xi_j = 0$.
\end{proof}

\begin{remark}
We note that because the operator $g(s)$ is non-linear, the weak limit $v$ of the testing sequence $v_n = \chi(\mx) g(u_n)$
does not generally coincide with the composition of the limit, $\chi(\mx) g(u)$. Consequently, when the defect measure $D$
is non-zero, the macroscopic reference state $Q(\mx) u v$ cannot be evaluated from the macroscopic limit $u$ alone.
To compute the physical defect $\auf \omega - Q(\mx) \chi(\mx) u g(u), \phi \zu$, this H-distribution framework must be
coupled with the associated Young measure $\eta_\mx$ of the sequence $(u_n)$ to identify the non-linear weak limit
$v(\mx) = \chi(\mx) \int_{\mathbf{R}} g(\lambda) \, d\eta_\mx(\lambda)$.
\end{remark}

\begin{remark}
The physical implications of this geometric restriction are specific to $\mathrm{VMO}$ environments. Because the defect
measure is supported orthogonal to the velocity symbol, the macroscopic reaction defect cannot propagate along the
macroscopic streamlines $\mathbf{a}(\mx)$. Any anomalous reaction rates generated by the highly oscillatory pollutant
sequence interacting with the sharply varying geological strata are trapped locally, accumulating transversally to the fluid flow.
\end{remark}

\section{Application: Zero-Order Operators and Cross-Phase Energy}\label{sec:zero_order}
To extend the $\mathrm{VMO}$ H-distribution framework further, we go beyond pointwise nonlinearities to evaluate non-local pseudo-differential operators.
While fractional integration operators inherently smooth weakly converging sequences (granting local strong compactness via the
Rellich-Kondrachov embedding), zero-order Fourier multipliers do not confer any gain in regularity, ensuring the sequences remain purely oscillatory.

\subsection{Physical Motivation and VMO Reactive Energy}
In wave mechanics and electromagnetism, the macroscopic transport of energy is frequently coupled with orthogonal, phase-shifted interactions.
For example, applying a zero-order operator such as the Riesz transform to an oscillating electric or acoustic
field corresponds to a precise $\pi/2$ phase shift in frequency space.
The interaction between the primary field and its phase-shifted projection represents non-local ``reactive'' or
``cross-phase'' energy, that is, energy that is temporarily trapped and exchanged locally within the medium rather
than being actively propagated or dissipated.
When waves propagate through highly heterogeneous environments, such as composite dielectrics or acoustic metamaterials,
the background material properties (e.g., electrical permittivity or bulk modulus) are governed by the coefficients $a_j(\mx)$.
As established in previous sections, the sharp macroscopic boundaries of these composite phases render the coefficients sharply
varying yet of vanishing mean oscillation, which is what requires the $\mathrm{L}^\infty(\Rd) \cap \mathrm{VMO}(\Rd)$ setting.
In classical linear theory with uniformly continuous coefficients, macroscopic reactive cross-energy can often be cleanly averaged or decoupled.
However, in high-contrast $\mathrm{VMO}$ environments, incident waves scatter strongly against the dense microscopic boundary layers.
Because zero-order operators preserve the spectral scaling of the original oscillations without conferring any fractional smoothing,
the primary wave and its phase-shifted counterpart remain coupled at the microscopic level.
The framework allows us to quantify this phenomenon.
By evaluating the H-distribution against a zero-order symbol, we can isolate the reactive energy trapped within the $\mathrm{VMO}$
structural heterogeneities, energy that would be invisible to classical macroscopic homogenization.

\subsection{The Cross-Phase Defect Measure}
Following the canonical framework, we use a zero-order Fourier multiplier $\mathcal{A}_m$ with symbol $m \in \mathrm{C}^\kappa(\mathrm{S}^{d-1})$.
The interaction between a primary sequence $(u_n)$ and its phase-shifted projection forms the non-local cross-energy density: $Q(\mx) u_n \mathcal{A}_m(u_n)$.

\begin{theorem}\label{thm:zero_order_app}
Let $p \in (2, \infty)$. Assume that $(u_n)$ is bounded in $\mathrm{L}^p(\Rd)$ and $u_n \rightharpoonup u$ weakly in $\mathrm{L}^p(\Rd)$.
Assume that $(u_n)$ satisfies the differential constraint:
\begin{equation*}
\sum_{j=1}^d \partial_j(a_j u_n) = f_n
\end{equation*}
where $f_n \to 0$ strongly in $\mathrm{W}^{-1, p}(\Rd)$ and the coefficients satisfy $a_j \in \mathrm{L}^\infty(\Rd) \cap \mathrm{VMO}(\Rd)$.
Let $m \in \mathrm{C}^\kappa(\mathrm{S}^{d-1})$ be a zero-order symbol, and let $\chi \in \mathrm{C}_c^\infty(\Rd)$ be a non-negative spatial cut-off.
Define the localized testing sequence $v_n = \chi(\mx) \mathcal{A}_m(u_n)$, and assume $v_n \rightharpoonup v = \chi \mathcal{A}_m(u)$ weakly in $\mathrm{L}^p(\Rd)$.
For any coefficient $Q \in \mathrm{L}^\infty(\Rd) \cap \mathrm{VMO}(\Rd)$, if the cross-energy sequence
$Q(\mx) u_n \overline{v_n} \to \omega$ in the sense of distributions $\mathcal{D}'(\Rd)$, then the
macroscopic cross-energy defect is exactly quantified by the H-distribution acting on any test function
$\phi \in \mathrm{C}_c^\infty(\Rd)$ as:
\begin{equation*}
\auf \omega - Q(\mx) u \overline{v}, \phi \zu = D(Q\phi, 1).
\end{equation*}
Furthermore, this cross-energy defect is geometrically trapped, strictly satisfying the microlocal constraint:
\begin{equation*}
\sum_{j=1}^d D(a_j \varphi, \kappa_j \psi) = 0
\end{equation*}
for all $\varphi \in \mathrm{L}^\infty(\Rd) \cap \mathrm{VMO}_c(\Rd)$ and $\psi \in \mathrm{C}^\kappa(\mathrm{S}^{d-1})$.
\end{theorem}

\begin{proof}
Because $p > 2$, we symmetrically set $q = p$, so that $\frac{1}{p} + \frac{1}{q} = \frac{2}{p} < 1$ strictly satisfies
the integrability gap. It suffices to verify that $(v_n)$ is an admissible testing sequence in $\mathrm{L}^p_c(\Rd)$
(Definition \ref{def:admissible}); both conclusions then follow from the Unified Macroscopic Defect Theorem
(Theorem \ref{thm:unified_defect}). By the classical Calder\'on-Zygmund theorem \cite{SteinHA}, the zero-order operator
$\mathcal{A}_m$ acts continuously on $\mathrm{L}^p(\Rd)$, so $\mathcal{A}_m(u_n)$ is uniformly bounded in $\mathrm{L}^p(\Rd)$;
multiplying by the compactly supported cut-off $\chi(\mx)$ gives $v_n \in \mathrm{L}^p_c(\Rd)$ with support in
$\mathrm{supp}\,\chi$, uniformly bounded and converging weakly to $v = \chi \mathcal{A}_m(u)$, hence admissible.
Because the operator is of order zero it grants no fractional Sobolev regularity, so the sequence retains its
microlocal oscillations and does not converge strongly. Theorem \ref{thm:unified_defect} then yields both the
exact cross-energy defect representation $\auf \omega - Q(\mx) u \overline{v}, \phi \zu = D(Q\phi, 1)$ and the
characteristic confinement $\sum_{j=1}^d D(a_j \varphi, \kappa_j \psi) = 0$, quantifying the geometric structure
of the energy trapped in the non-local phase shifts generated by the $\mathrm{VMO}$ medium.
\end{proof}

\section{Application: High-Contrast Acoustics and Sub-Critical Energy}\label{sec:acoustics}
To use the last of the canonical testing sequences introduced in Section \ref{sec:canonical}, we turn to the sub-critical Nemytskii operator.
This formulation arises in the study of high-frequency acoustic wave propagation through highly heterogeneous
geological or manufactured media \cite{Ammari}, where uniform continuity assumptions must be discarded.

\subsection{Acoustic Scattering and Impedance in VMO Strata}
Consider high-frequency acoustic pressure fluctuations traveling through a highly stratified geological formation (e.g., seismic waves in the Earth's crust) or engineered acoustic metamaterials.
In classical continuous media, acoustic waves propagate with adiabatic impedance transitions, allowing the energy to be tracked smoothly via ray theory.
However, in highly stratified environments, the medium's density and bulk modulus vary abruptly, through thin but continuous transition layers, across geological boundaries or structural interfaces.
While these transitions are continuous at the microscopic level, their macroscopic profile is very steep.
Classical $\mathrm{C}^0$ approximations fail because the spatial gradient of the acoustic impedance (the product of density and wave speed) effectively blows up.
The functional space $\mathrm{L}^\infty(\Rd) \cap \mathrm{VMO}(\Rd)$ models these sharp impedance interfaces without requiring artificial smoothing.
Let the scalar sequence $(u_n)$ represent the highly oscillatory acoustic pressure field, constrained by the stationary momentum balance
$\sum_{j=1}^d \partial_j(a_j u_n) = f_n$, where the coefficients $a_j(\mx)$ represent the sharply varying background material properties.
In these high-contrast $\mathrm{VMO}$ environments, incident acoustic energy does not merely propagate;
it undergoes local scattering, resonance, and reflection, causing the acoustic pressure to oscillate strongly and become trapped within the microscopic transition layers.
To quantify the energy of these trapped waves, we must look beyond linear acoustics.
In regions of high pressure concentration, nonlinear acoustic effects such as wave steepening, localized harmonic generation, and nonlinear attenuation become significant.
To model this nonlinear energy buildup without violating the integrability gap of the H-distribution framework, we employ a sub-critical polynomial functional.
Let the nonlinear acoustic response be governed by $g(s) = s|s|^{\gamma-1}$ for some growth exponent $\gamma > 1$.
Because $g$ is continuous and strictly monotonic, it captures the nonlinear pressure amplification.
By restricting the growth to a sub-critical threshold, we prevent unphysical shock-wave singularities in the weak formulation,
keeping the localized energy within the target Lebesgue space required by the $\mathrm{L}^p-\mathrm{L}^q$ framework.

\begin{theorem}\label{thm:acoustic_app}
Let $p, q \in (1, \infty)$ such that $\frac{1}{p} + \frac{1}{q} < 1$.
Assume that $(u_n)$ is bounded in $\mathrm{L}^p(\Rd)$ and $u_n \rightharpoonup u$ weakly in $\mathrm{L}^p(\Rd)$.
Assume that $(u_n)$ satisfies the acoustic constraint:
\begin{equation*}
\sum_{j=1}^d \partial_j(a_j u_n) = f_n
\end{equation*}
where $f_n \to 0$ strongly in $\mathrm{W}^{-1, p}(\Rd)$ and the coefficients satisfy $a_j \in \mathrm{L}^\infty(\Rd) \cap \mathrm{VMO}(\Rd)$.
Let the energy response exponent satisfy $\gamma \leq \frac{p}{q}$. Define the localized sub-critical testing sequence
$v_n = \chi(\mx) u_n |u_n|^{\gamma-1}$ for a non-negative spatial cut-off $\chi \in \mathrm{C}_c^\infty(\Rd)$, and
assume $v_n \rightharpoonup v$ weakly in $\mathrm{L}^q(\Rd)$.
For any coefficient $Q \in \mathrm{L}^\infty(\Rd) \cap \mathrm{VMO}(\Rd)$, if the sub-critical energy sequence
$Q(\mx) u_n v_n \to \omega$ in the sense of distributions $\mathcal{D}'(\Rd)$, then the macroscopic sub-critical defect
$\omega - Q(\mx) u v$ is exactly quantified by the H-distribution acting on any test function $\phi \in \mathrm{C}_c^\infty(\Rd)$ as:
\begin{equation*}
\auf \omega - Q(\mx) u v, \phi \zu = D(Q\phi, 1).
\end{equation*}
Furthermore, this non-linear energy defect is geometrically trapped by the $\mathrm{VMO}$ strata, strictly satisfying the microlocal impedance constraint:
\begin{equation*}
\sum_{j=1}^d D(a_j \varphi, \kappa_j \psi) = 0
\end{equation*}
for all $\varphi \in \mathrm{L}^\infty(\Rd) \cap \mathrm{VMO}_c(\Rd)$ and $\psi \in \mathrm{C}^\kappa(\mathrm{S}^{d-1})$.
\end{theorem}

\begin{proof}
It suffices to verify that the localized sub-critical sequence $(v_n)$ is an admissible testing sequence in
$\mathrm{L}^q_c(\Rd)$ (Definition \ref{def:admissible}); both conclusions then follow from the Unified Macroscopic
Defect Theorem (Theorem \ref{thm:unified_defect}). Let $g(s) = s|s|^{\gamma-1}$ and $K = \mathrm{supp}\,\chi$.
Because $u_n$ is bounded in $\mathrm{L}^p(\Rd)$, its restriction to $K$ belongs to $\mathrm{L}^p(K)$, and the
sub-critical growth bound $|g(s)| = |s|^\gamma$ with $\gamma \leq p/q$ gives $\gamma q \leq p$, so the bounded-domain
embedding $\mathrm{L}^p(K) \hookrightarrow \mathrm{L}^{\gamma q}(K)$ acts continuously and
\begin{equation*}
\int_K |\chi g(u_n)|^q \, d\mx \leq \|\chi\|_\infty^q C_K \|u_n\|_{\mathrm{L}^p(K)}^{\gamma q} < \infty.
\end{equation*}
Thus $v_n = \chi u_n |u_n|^{\gamma-1}$ is uniformly bounded in $\mathrm{L}^q(\Rd)$ with support in $K$, safely
bypassing the failure of global integrability, and together with the assumed weak limit $v$ it is admissible.
Theorem \ref{thm:unified_defect} then delivers both the exact sub-critical defect representation
$\auf \omega - Q(\mx) u v, \phi \zu = D(Q\phi, 1)$ and the characteristic confinement $\sum_{j=1}^d D(a_j \varphi, \kappa_j \psi) = 0$,
so the sub-critical non-linear defect is trapped in the null space of the acoustic impedance symbol.
\end{proof}

\begin{remark}
Similarly, because the sub-critical energy response $g(s) = s|s|^{\gamma-1}$ is non-linear, the weak limit $v$
of the testing sequence $v_n = \chi(\mx) u_n |u_n|^{\gamma-1}$ does not generally coincide with the composition
of the limit, $\chi(\mx) u |u|^{\gamma-1}$. When the acoustic defect measure $D$ is non-zero, the macroscopic
reference state $Q(\mx) u v$ cannot be isolated from the macroscopic limit $u$ alone. To compute the physical
trapped energy, this framework must be coupled with the associated Young measure $\eta_\mx$ of the sequence
$(u_n)$ to evaluate the non-linear weak limit $v(\mx) = \chi(\mx) \int_{\mathbf{R}} \lambda|\lambda|^{\gamma-1} \, d\eta_\mx(\lambda)$.
\end{remark}

\begin{remark}
This result shows that even though the acoustic response is non-linear (governed by the sub-critical polynomial energy buildup),
the resulting microscopic energy scattering is not geometrically chaotic.
The $\mathrm{VMO}$ framework shows that this non-linear trapped energy remains constrained by the linear microlocal geometry of
the underlying geological strata, oscillating orthogonal to the background acoustic impedance vector $\mathbf{a}(\mx)$.
\end{remark}

\section{Future Directions and Possible Applications}\label{sec:future}
The framework for $\mathrm{VMO}$ H-distributions established in this note opens several directions for further analysis.
The main feature of this framework is the ability to handle rough, sharply varying coefficients that classical H-measures and
distributions cannot process, provided their oscillation vanishes in the mean, i.e.\ they are of vanishing mean oscillation.
Some possible future applications are the following.

\subsection{Fractional Derivatives and Non-Local Dirichlet Forms}
While fractional integration smooths weakly converging sequences, fractional derivatives strictly preserve and amplify
microlocal oscillations. If a sequence is bounded in a fractional Sobolev space $\mathrm{W}^{s, p}(\Rd)$, one can construct
a non-local testing sequence via the fractional Laplacian: $v_n = \chi (-\Delta)^{\beta/2} u_n$. Because the compact-support
construction is compatible with fractional Sobolev regularity without generating unphysical boundary singularities
(Proposition \ref{prop:frac_sobolev} in Appendix \ref{sec:appendix_topology}), provided the fractional order satisfies the
integrability gap $s - \beta \geq 0$ such that $v_n \in \mathrm{L}^q_c(\Rd)$, the $\mathrm{VMO}$ H-distribution framework
can be used to quantify non-local phase transitions and fractional Dirichlet energy defects in highly heterogeneous composite media.

\subsection{Variable-Exponent Nemytskii Operators}
The sub-critical polynomial bounds established in Section \ref{subsec:subcrit} can be generalized to spatially dependent non-linearities.
By defining the testing sequence $v_n = \chi(\mx) u_n |u_n|^{\gamma(\mx)-1}$ for a variable exponent
$\gamma(\mx) \in \mathrm{L}^\infty(\Rd)$, the sequence naturally maps into Orlicz-Lebesgue spaces.
Provided the essential supremum satisfies $\esssup \gamma(\mx) \leq p/q$, the framework remains intact.
This gives a pathway for using H-distributions to analyze electrorheological non-Newtonian fluids and smart materials modeled by $p(x)$-Laplacians in $\mathrm{VMO}$ domains.

\subsection{Homogenization of Composite Media with Rough Coefficients}
The homogenization of differential operators is a classical framework for computing the effective macroscopic
properties of highly heterogeneous materials \cite{ZKO, Milton}.
In modern approaches, standard H-measures are frequently used to quantify the associated oscillation and concentration effects.
However, the classical H-measure theory requires continuous coefficients, which poses physical limitations when
distinct composite phases form sharp boundaries.
Because $\mathrm{VMO}$ permits rough coefficients whose mean oscillation vanishes at small scales, the Localisation
Principle derived in Theorem \ref{thm:localisation} can be applied to the homogenization of elliptic PDEs modeling
composite materials or stratified fluids where the microscopic heterogeneities belong to $\mathrm{VMO}$.

\subsection{Vectorial Compensated Compactness and a div--curl Lemma}
The scalar Localisation Principle of Theorem \ref{thm:localisation} constrains defects generated by a single
divergence-form equation. A natural and important extension is a genuine div--curl lemma for vector fields with
$\mathrm{VMO}$ coefficients: given $u_n\rightharpoonup u$ and $v_n\rightharpoonup v$ with $\mathrm{div}(A u_n)$
and $\mathrm{curl}(B v_n)$ precompact in the appropriate negative Sobolev spaces and
$A,B\in\mathrm{L}^\infty(\Rd)\cap\mathrm{VMO}(\Rd)$, one would seek to pass the bilinear product $u_n\cdot v_n$
to its weak limit. The coefficient roughness is not the obstruction here, since it is again absorbed through the
commutator mechanism of Lemma \ref{1stcommlemma_Lp}. Three genuine difficulties must instead be addressed.
First, the present construction is scalar, whereas div--curl is intrinsically vectorial, so a matrix-valued
H-distribution $D_{ij}$ associated with the components $(u_n^i, v_n^j)$ is required. Second, the Localisation
Principle currently confines only the first slot; a symmetric two-sided localisation is needed, for which the
Hermitian symmetry of Proposition \ref{prop:hermitian} provides the appropriate transfer between the two slots.
Third, and most fundamentally, the classical div--curl pairing $u_n\cdot v_n$ is critical, corresponding to the
endpoint $\frac{1}{p}+\frac{1}{q}=1$, which lies precisely outside the admissible range of the present framework:
the strict gap $\frac{1}{p}+\frac{1}{q}<1$ that underpins the localized John--Nirenberg estimate is incompatible
with the borderline integrability. Consequently, the current construction yields only a \emph{sub-critical}
div--curl lemma, in which the product is more integrable than $\mathrm{L}^1(\Rd)$; recovering the sharp endpoint
statement would require extending the H-distribution construction to the excluded boundary case, where the
$\mathrm{L}^\infty$-versus-$\mathrm{BMO}$ control obstruction that motivates the strict gap resurfaces.

\subsection{Velocity Averaging in Kinetic Theory}
Corollary \ref{cor:Lp_Linfty} establishes the $\mathrm{L}^p-\mathrm{L}^\infty$ pairing, which provides the
functional topology required for velocity averaging lemmas in kinetic equations.
In standard transport equations, macroscopic observables are often strongly compact even if the kinetic density only converges weakly.
This framework can be adapted to transport equations where the macroscopic velocity or advection field lacks
continuity and exhibits $\mathrm{VMO}$ regularity (cf.\ \cite{EMM}), bypassing the regularity conditions that
impede classical compactness arguments.
We stress, however, that a genuine averaging lemma requires establishing a \emph{gain of regularity} for the
velocity average, which is not furnished by the defect representation alone but demands an additional
interpolation estimate; the discontinuous-flux case in particular has already been developed in \cite{EMM},
with which any such extension would substantially overlap.

\subsection{Non-Linear Zener Breakdown and Orlicz-VMO Spaces}
While the framework developed in this paper quantifies energy defects for linear and sub-critical non-linear systems,
highly heterogeneous engineered materials often exhibit critical, spatially varying non-linearities. One candidate for
future research is the evaluation of localized Joule heating in polycrystalline ceramic varistors (e.g., ZnO grains
separated by ultra-thin $\mathrm{Bi}_2\mathrm{O}_3$ boundaries) \cite{Cimatti}. 

Prior to Zener breakdown, the microscopic boundaries act as strict insulators, trapping energy. During breakdown,
the constitutive electromagnetic relationship follows a highly non-linear, variable-exponent law:
$\mathbf{J}_n = \sigma(\mx) |\mathbf{E}_n|^{\alpha(\mx)-1}\mathbf{E}_n$. Because the growth exponent $\alpha(\mx)$
varies drastically between the ohmic interior of the grain ($\alpha = 1$) and the non-linear interface ($\alpha \gg 1$),
the sequences do not natively reside within standard Lebesgue spaces $\mathrm{L}^p(\Rd)$.

To quantify the macroscopic defect measures generated by these sharp, lower-dimensional boundaries, the current
$\mathrm{L}^p-\mathrm{L}^q$ framework must be generalized to variable-exponent Lebesgue spaces $\mathrm{L}^{p(\mx)}(\Rd)$
or generalized Musielak-Orlicz spaces $\mathrm{L}^\Phi(\Rd)$; the homogenization of the underlying non-linear
conductivity problem has been studied in related settings \cite{Wellander}. Extending the $\mathrm{VMO}$
Macroscopic Defect Representation to these topologies would allow the evaluation of energy localization in materials
with strong, spatially-dependent non-linear growth.

\subsection{Non-Linear Evaluation via Young Measure Coupling}
In the non-linear applications in Sections \ref{sec:transport} and \ref{sec:acoustics}, the $\mathrm{VMO}$
H-distribution framework isolated the geometric support of the macroscopic defect measures. However, because
weak limits are not preserved under continuous non-linear transformations (e.g., $g(u_n) \rightharpoonup v \neq g(u)$),
the physical magnitude of the trapped energy cannot be evaluated from the macroscopic limit $u$ alone. 

To pass from identifying the geometric structure of the defect to evaluating its scalar magnitude, the
$\mathrm{VMO}$ H-distribution framework must be coupled with the theory of Young measures \cite{LazarMitrovic, MM}.
By associating a parametrized probability measure $\eta_\mx$ to the primary sequence $(u_n)$, the non-linear weak
limit can be evaluated as an expected value: $v(\mx) = \int_{\mathbf{R}} g(\lambda) \, d\eta_\mx(\lambda)$.
Establishing the commutation and compatibility conditions between microlocal $\mathrm{VMO}$ H-distributions and
macroscopic Young measures is a next step in the treatment of high-contrast non-linear energy defects.

\section{Acknowledgements}
The author is deeply grateful to Professor Luc Tartar for stimulating discussions concerning this topic and the role of Hardy spaces during his visit to the University of Zagreb in 2016.
A part of the work was performed while the author was visiting University Paris-Sud XI under the scholarship of the Government of the French Republic, whose support he gratefully acknowledges.
He thanks Laurent Moonens for the hospitality.

\section*{Declarations}
\textbf{Funding:} This research is supported by the Croatian Science Foundation, project number 9780, and by the University of Zagreb through grant PP04/2016. \\
\textbf{Conflict of interest:} The author declares that he has no conflict of interest. \\
\textbf{Data availability:} Data sharing is not applicable to this article as no datasets were generated or analysed during the current study.\\
\textbf{Generative AI and AI-assisted technologies:} During the preparation of this work, the author(s) used Google Gemini as an assistive tool
to transcribe handwritten mathematical notes into \LaTeX{} formatting, to polish the English prose for readability, and to help draft the initial
abstract and manuscript summary. Additionally, the AI was utilized during the research phase for conceptual exploration, specifically to search
heuristically for potential counterexamples to stress-test preliminary hypotheses. After using this tool, the author(s) meticulously reviewed,
verified, and edited all generated text and code. All mathematical claims, proofs, and counterexamples were independently rigorously verified
by the human author(s). The author(s) take full intellectual responsibility for the final content of this publication, including all mathematical
proofs, formatting, and conceptual framing.

\appendix
\section{Topological Properties of the LF-Space $\mathrm{VMO}_c$}
\label{sec:appendix_topology}
The definition of $\mathrm{VMO}_c(\Rd)$ as a strict inductive limit of Banach spaces (an LF-space) yields several topological properties
that bear on the behavior of weakly converging sequences in highly heterogeneous media. We establish the relevant properties here to
provide a foundation for the functional evaluations in the main text.

\subsection{Independence of the Exhausting Sequence}
The construction of $\mathrm{VMO}_c(\Rd)$ in Section \ref{sec:main_result} fixes a compact exhaustion $(K_n)_n$. We first record that
neither the topology nor the resulting H-distribution depends on this choice.

\begin{proposition}\label{prop:exhaustion}
The inductive limit topology on $\mathrm{VMO}_c(\Rd)$, and hence the H-distribution $D$ of Theorem \ref{thm:h_dist_vmo}, is independent
of the choice of compact exhaustion $(K_n)_n$.
\end{proposition}

\begin{proof}
Let $(K_n)_n$ and $(K_n')_n$ be two compact exhaustions of $\Rd$. Since each $K_n$ is compact and $\Rd=\bigcup_m \mathrm{Int}\,K_m'$,
there is an index $m=m(n)$ with $K_n\subset\mathrm{Int}\,K_m'$, and symmetrically; thus the two families of Banach steps
$\{\mathrm{VMO}_{K_n}\}$ and $\{\mathrm{VMO}_{K_m'}\}$ are mutually cofinal, and the inclusions between them are isometric embeddings
of Banach spaces. A strict inductive limit is unchanged, as a locally convex space and in its topology, under passage to a cofinal
subfamily; hence the two constructions yield the same space $\mathrm{VMO}_c(\Rd)$ with the same topology. Finally, by the
joint-continuity characterization established below (Proposition \ref{prop:joint_cont}) the functional $D$ is determined by its
restrictions to the Banach steps $\mathrm{VMO}_K(\Rd)$, which are intrinsic to the compact sets $K$ and not to their enumeration;
therefore $D$ is likewise independent of the exhaustion.
\end{proof}

\subsection{Failure of the Montel Property}
In standard distribution theory, the space of smooth test functions $\mathcal{D}(\Rd)$ is a Montel space, meaning that every closed
and bounded subset is relatively compact. This guarantees that bounded sequences of test functions admit strongly convergent subsequences.
The space of compactly supported $\mathrm{VMO}$ functions strictly fails this property.

\begin{proposition}\label{prop:montel}
The strict LF-space $\mathrm{VMO}_c(\Rd)$ is not a Montel space; a fortiori, neither is its bounded subspace $\mathrm{L}^\infty(\Rd)\cap\mathrm{VMO}_c(\Rd)$.
\end{proposition}

\begin{proof}
A strict inductive limit of Fr\'echet spaces is a Montel space if and only if each of its generating step spaces is a Montel space.
By F.~Riesz's lemma, the closed unit ball of an infinite-dimensional normed space is never compact, so no infinite-dimensional Banach space is a Montel space.

Let $K \subset \Rd$ be a compact set with non-empty interior. The step space $\mathrm{VMO}_K(\Rd)$ is an infinite-dimensional Banach space,
since it already contains the infinite-dimensional family $\mathrm{C}_c^\infty(\mathrm{Int}\,K)$, and hence is not Montel.
Since the generating steps are not Montel spaces, the strict inductive limit $\mathrm{VMO}_c(\Rd)$ is not a Montel space, and its bounded
sets need not be relatively compact. The same conclusion transfers to the dense subspace $\mathrm{L}^\infty(\Rd)\cap\mathrm{VMO}_c(\Rd)$,
which likewise contains $\mathrm{C}_c^\infty(\mathrm{Int}\,K)$.
\end{proof}

\begin{remark}
In practice, this failure means that one cannot extract strongly converging subsequences from merely bounded sequences within this test space.
This forces the use of weak-$\ast$ compactness arguments and is the reason for the H-distribution framework developed in Section \ref{sec:main_result},
as the macroscopic defect measures must absorb the persistent microlocal oscillations that do not converge strongly.
\end{remark}

\subsection{Diameter Dependence of the Continuity Constant}
The continuity constant of Theorem \ref{thm:h_dist_vmo} is controlled by the diameter of the support, equivalently by the radius of the smallest
enclosing ball, which is the quantitative reason why the inductive limit topology, rather than a single global $\mathrm{BMO}$ norm, is unavoidable.

\begin{proposition}\label{prop:JN_constant}
Let $r\in(1,\infty)$ and let $\varphi\in\mathrm{VMO}(\Rd)$ be supported in a compact set $K$. Then
\begin{equation}\label{eq:JN_bound}
\|\varphi\|_{\mathrm{L}^r(K)} \leq C_d\, r\, (\operatorname{diam}K)^{d/r}\, \|\varphi\|_{\mathrm{BMO}(\Rd)},
\end{equation}
with $C_d$ depending only on the dimension. Consequently the bounding constant $C_{K_\varphi}$ of Theorem \ref{thm:h_dist_vmo}, with
$\frac1r = 1-\frac1p-\frac1q$, obeys $C_{K_\varphi}\lesssim (\operatorname{diam}K_\varphi)^{d/r}$ and tends to infinity as $\operatorname{diam}K_\varphi\to\infty$.
\end{proposition}

\begin{proof}
Write $\omega_d=|B(0,1)|$, and let $R_K$ denote the radius of the smallest closed ball enclosing $K$, so that $R_K\leq\operatorname{diam}K$
and $|K|\leq\omega_d R_K^d\leq\omega_d(\operatorname{diam}K)^d$. Let $B$ be the closed ball concentric with this smallest enclosing ball and of radius
\begin{equation*}
\rho=\max\!\big(R_K,\,(2|K|/\omega_d)^{1/d}\big).
\end{equation*}
Then $B\supseteq K$, and $|B|=\omega_d\rho^d\geq 2|K|$, so that
\begin{equation*}
|B\setminus K|=|B|-|K|\geq\tfrac12|B|.
\end{equation*}
Moreover $(2|K|/\omega_d)^{1/d}\leq 2^{1/d}\operatorname{diam}K$ and $R_K\leq\operatorname{diam}K$, whence $\rho\leq 2^{1/d}\operatorname{diam}K$ and therefore
\begin{equation*}
|B|=\omega_d\rho^d\leq 2\omega_d(\operatorname{diam}K)^d.
\end{equation*}

Because $\varphi$ vanishes on $B\setminus K$, the mean $\varphi_B$ satisfies
\begin{equation*}
|\varphi_B|\,\frac{|B\setminus K|}{|B|}=\frac{1}{|B|}\int_{B\setminus K}|\varphi_B|\,d\mx=\frac{1}{|B|}\int_{B\setminus K}|\varphi-\varphi_B|\,d\mx\leq\frac{1}{|B|}\int_{B}|\varphi-\varphi_B|\,d\mx\leq\|\varphi\|_{\mathrm{BMO}(\Rd)},
\end{equation*}
and $|B\setminus K|\geq\tfrac12|B|$ yields $|\varphi_B|\leq 2\|\varphi\|_{\mathrm{BMO}(\Rd)}$. The John-Nirenberg inequality, applied on the single ball $B$, gives
\begin{equation*}
\|\varphi-\varphi_B\|_{\mathrm{L}^r(B)}\leq C_d\, r\,|B|^{1/r}\,\|\varphi\|_{\mathrm{BMO}(\Rd)}.
\end{equation*}
Since $K\subseteq B$, the triangle inequality in $\mathrm{L}^r(K)$ and $|K|\leq|B|$ give
\begin{equation*}
\|\varphi\|_{\mathrm{L}^r(K)}\leq\|\varphi-\varphi_B\|_{\mathrm{L}^r(B)}+|\varphi_B|\,|K|^{1/r}\leq\big(C_d\, r+2\big)\,|B|^{1/r}\,\|\varphi\|_{\mathrm{BMO}(\Rd)}.
\end{equation*}
Finally, $|B|^{1/r}\leq(2\omega_d)^{1/r}(\operatorname{diam}K)^{d/r}$ and $C_d\, r+2\leq(C_d+2)r$, so that
\begin{equation*}
\|\varphi\|_{\mathrm{L}^r(K)}\leq C_d'\, r\,(\operatorname{diam}K)^{d/r}\,\|\varphi\|_{\mathrm{BMO}(\Rd)},
\end{equation*}
with $C_d'=(C_d+2)\max\!\big(1,2\omega_d\big)$ depending only on the dimension. The stated dependence of $C_{K_\varphi}$
then follows from the derivation of the bound in Theorem \ref{thm:h_dist_vmo}.
\end{proof}

\begin{remark}\label{rem:tubular}
The appearance of the diameter, rather than the Lebesgue measure $|K|$, is intrinsic to the argument: any ball $B$ containing $K$
has radius at least $R_K\geq\tfrac12\operatorname{diam}K$, hence $|B|\geq\omega_d(\tfrac12\operatorname{diam}K)^d$, however small $|K|$ may be.
For families that are thin relative to their diameter this gap is unbounded. For instance the tubular neighbourhoods
$K_n=[0,n]\times[-\varepsilon,\varepsilon]^{d-1}$ (with $\varepsilon>0$ fixed) satisfy $|K_n|=2^{d-1}\varepsilon^{d-1}n$, while every
enclosing ball has radius $\geq n/2$ and thus measure $\gtrsim n^d$; the ratio $|B|/|K_n|$ diverges as $n\to\infty$. Consequently no
estimate of the form $\|\varphi\|_{\mathrm{L}^r(K)}\leq C_d\, r\,|K|^{1/r}\|\varphi\|_{\mathrm{BMO}(\Rd)}$ with a purely dimensional
constant can be obtained by the method above, and the volume-based formulation of an earlier version of this work must be replaced by
the diameter-based one. We thank the anonymous referee for pointing out this defect.
\end{remark}

\begin{remark}
This makes precise the third structural prerequisite of Remark \ref{rem:prerequisites}: because $C_{K_\varphi}$ grows with the diameter
of the support, $D$ is genuinely discontinuous for the global $\mathrm{BMO}(\Rd)$ norm, and only the localized boundedness afforded by
the strict LF-space topology renders it continuous.
\end{remark}

\begin{conjecture}\label{conj:volume}
It remains open to us whether the Lebesgue measure may, after all, replace the diameter in Proposition \ref{prop:JN_constant}.
The counterexample of Remark \ref{rem:tubular} defeats the enclosing-ball proof, but it does not exhibit a single function $\varphi$
that violates the inequality itself. We conjecture that the volume-form estimate
\begin{equation*}
\|\varphi\|_{\mathrm{L}^r(K)}\leq C_d\, r\,|K|^{1/r}\,\|\varphi\|_{\mathrm{BMO}(\Rd)}
\end{equation*}
nevertheless holds for every compact $K\subset\Rd$ and every $\varphi\in\mathrm{VMO}(\Rd)$ supported in $K$, with $C_d$ depending only
on the dimension, so that the diameter in Proposition \ref{prop:JN_constant} is an artifact of the method rather than of the estimate.
We can at present neither prove this nor produce a counterexample: on any support thin enough to separate $|K|$ from $(\operatorname{diam}K)^d$,
the zero-extension of $\varphi$ across the unavoidably large boundary already forces $\|\varphi\|_{\mathrm{BMO}(\Rd)}\gtrsim\|\varphi\|_{\mathrm{L}^\infty}$,
and in every example we have examined this growth of the seminorm offsets the apparent gain in $|K|^{1/r}$. Should the conjecture hold,
Proposition \ref{prop:JN_constant} and every continuity statement depending on it would remain valid verbatim, with $(\operatorname{diam}K)^{d/r}$
replaced by $|K|^{1/r}$; should it fail, the diameter dependence established above is sharp.
\end{conjecture}

\subsection{Compatibility with Fractional Sobolev Regularity}
The strict LF construction is compatible with fractional Sobolev regularity, providing the justification for evaluating non-local
fractional Dirichlet forms without generating unphysical boundary singularities. We stress at the outset that vanishing mean oscillation
does \emph{not} by itself confer any fractional Sobolev smoothness: there exist bounded $\mathrm{VMO}$ functions belonging to no
$\mathrm{W}^{s,p}_{loc}(\Rd)$ with $s > 0$, so no blanket inclusion $\mathrm{VMO}_c(\Rd) \hookrightarrow \mathrm{W}^{s,p}_{loc}(\Rd)$
can hold. What the compact-support construction does guarantee is the absence of the zero-extension pathologies that afflict fractional
Sobolev spaces on bounded domains.

\begin{proposition}\label{prop:frac_sobolev}
Let $s \in (0, 1)$ and $p \in (1, \infty)$, and let $K \subset \Rd$ be compact. Equip $\mathrm{VMO}_K(\Rd) \cap \mathrm{W}^{s,p}(\Rd)$
with the graph norm $\|\cdot\|_{\mathrm{BMO}(\Rd)} + \|\cdot\|_{\mathrm{W}^{s,p}(\Rd)}$. Then the identity map
\begin{equation*}
\iota: \mathrm{VMO}_K(\Rd) \cap \mathrm{W}^{s,p}(\Rd) \hookrightarrow \mathrm{W}^{s,p}(\Rd)
\end{equation*}
is a continuous inclusion, and the extension of such a function by zero to all of $\Rd$ coincides with the function itself.
\end{proposition}

\begin{proof}
A function $\varphi$ in the domain is supported in the fixed compact set $K$, so its trivial extension by zero to $\Rd$ is $\varphi$ itself;
no artificial interface is created at $\partial K$ and the global Gagliardo seminorm equals the intrinsic one. Continuity of $\iota$ is then
immediate from the definition of the graph norm, since
$\|\iota\varphi\|_{\mathrm{W}^{s,p}(\Rd)} = \|\varphi\|_{\mathrm{W}^{s,p}(\Rd)} \leq \|\varphi\|_{\mathrm{BMO}(\Rd)} + \|\varphi\|_{\mathrm{W}^{s,p}(\Rd)}$.
Consequently, global fractional operators defined via the Fourier transform act on such functions without boundary blow-up.
\end{proof}

\begin{remark}
Because the functions naturally decay to zero outside their compact support, extending them to $\Rd$ introduces no boundary jumps.
This ensures that global fractional differential operators, defined via the Fourier transform, can act continuously on sequences in
$\mathrm{VMO}_c(\Rd)$ without blowing up across the support boundaries.
\end{remark}

\subsection{Density of Bounded Functions in the Banach Steps}
The extension of the H-distribution $D$ from $\mathrm{L}^\infty(\Rd)\cap\mathrm{VMO}_c(\Rd)$, on which it is defined by the integral
formula, to the whole test space $\mathrm{VMO}_c(\Rd)$ rests on the density of the bounded functions within each Banach step. Because
a step space $\mathrm{VMO}_K(\Rd)$ contains genuinely unbounded functions (a $\mathrm{VMO}$ function may carry a local $\log\log$-type singularity),
this density is not a formality: it is not evident that an unbounded $\mathrm{VMO}$ function is approximable in $\mathrm{BMO}$ norm by its truncations,
since boundedness of the truncation operator on $\mathrm{BMO}$ (which is classical) is weaker than convergence of the truncations.
We supply a proof, in which the compact support plays a decisive role.

\begin{proposition}\label{prop:bounded_density}
For every compact set $K\subset\Rd$, the bounded functions $\mathrm{L}^\infty(\Rd)\cap\mathrm{VMO}_K(\Rd)$ are dense in the Banach step $\mathrm{VMO}_K(\Rd)$ in the $\mathrm{BMO}$ norm.
\end{proposition}

\begin{proof}
For $M>0$ let $T_M(t)=\max(-M,\min(M,t))$ be the truncation at level $M$; it is $1$-Lipschitz and satisfies $T_M(0)=0$. Given $f\in\mathrm{VMO}_K(\Rd)$, set $f_M=T_M\circ f$.

\emph{Step 1: $f_M\in\mathrm{L}^\infty(\Rd)\cap\mathrm{VMO}_K(\Rd)$.} By construction $\|f_M\|_{\mathrm{L}^\infty(\Rd)}\leq M$,
and $T_M(0)=0$ forces $\{f_M\neq 0\}\subseteq\{f\neq 0\}$, whence $\mathrm{supp}\,f_M\subseteq\mathrm{supp}\,f\subseteq K$.
For any ball $B$ and any constant $c$ one has the elementary estimate $\frac{1}{|B|}\int_B|g-g_B|\leq\frac{2}{|B|}\int_B|g-c|$;
applying it with $g=f_M$ and $c=T_M(f_B)$, and using the $1$-Lipschitz bound $|T_M(f)-T_M(f_B)|\leq|f-f_B|$, gives
\begin{equation}\label{eq:trunc_osc}
\frac{1}{|B|}\int_B|f_M-(f_M)_B|\,d\mx\leq\frac{2}{|B|}\int_B|f-f_B|\,d\mx.
\end{equation}
Thus the mean oscillation of $f_M$ over any ball is at most twice that of $f$; in particular $\|f_M\|_{\mathrm{BMO}(\Rd)}\leq 2\|f\|_{\mathrm{BMO}(\Rd)}$,
and since $f\in\mathrm{VMO}$ the modulus of mean oscillation of $f_M$ vanishes as the radius tends to zero. Hence $f_M\in\mathrm{VMO}_K(\Rd)$.

\emph{Step 2: $f_M\to f$ in $\mathrm{BMO}(\Rd)$.} Write $r_M=f-f_M$; then $r_M$ is supported in $K$ and $|r_M|=(|f|-M)_+$.
Let $\omega_f(t)=\sup_{|B|\leq t}\frac{1}{|B|}\int_B|f-f_B|$ be the mean-oscillation modulus of $f$, so that $\omega_f(t)\to 0$ as $t\to 0^+$
because $f\in\mathrm{VMO}$. Combining the seminorm triangle inequality with \eqref{eq:trunc_osc} applied to $f$ and to $f_M$, for every ball $B$,
\begin{equation}\label{eq:small_balls}
\frac{1}{|B|}\int_B|r_M-(r_M)_B|\,d\mx\leq\frac{1}{|B|}\int_B|f-f_B|+\frac{1}{|B|}\int_B|f_M-(f_M)_B|\leq 3\,\omega_f(|B|),
\end{equation}
uniformly in $M$. On the other hand, using $\frac{1}{|B|}\int_B|r_M-(r_M)_B|\leq\frac{2}{|B|}\int_B|r_M|$ together with the compact support of $r_M$,
\begin{equation}\label{eq:large_balls}
\frac{1}{|B|}\int_B|r_M-(r_M)_B|\,d\mx\leq\frac{2}{|B|}\int_{B\cap K}(|f|-M)_+\,d\mx\leq\frac{2}{|B|}\int_K(|f|-M)_+\,d\mx.
\end{equation}
Fix $\varepsilon>0$. By \eqref{eq:small_balls} choose $\delta>0$ with $3\,\omega_f(\delta)<\varepsilon$, so that the oscillation of $r_M$
over every ball with $|B|\leq\delta$ is below $\varepsilon$, uniformly in $M$. For balls with $|B|>\delta$, \eqref{eq:large_balls} bounds
the oscillation by $\frac{2}{\delta}\int_K(|f|-M)_+$. Since $f\in\mathrm{BMO}(\Rd)$ is locally integrable, $(|f|-M)_+\leq|f|\in\mathrm{L}^1(K)$
decreases to $0$ pointwise, and dominated convergence yields $\int_K(|f|-M)_+\to 0$ as $M\to\infty$; choose $M_0$ with
$\frac{2}{\delta}\int_K(|f|-M)_+<\varepsilon$ for $M\geq M_0$. Taking the supremum over all balls, $\|f-f_M\|_{\mathrm{BMO}(\Rd)}<\varepsilon$
for $M\geq M_0$. As each $f_M\in\mathrm{L}^\infty(\Rd)\cap\mathrm{VMO}_K(\Rd)$, the density follows.
\end{proof}

\begin{remark}
The compact support is essential in Step 2: it is what turns the tail term into the fixed integral $\int_K(|f|-M)_+$, which is finite and
vanishes as $M\to\infty$. Over an unbounded support the truncations $f_M$ need not converge to $f$ in $\mathrm{BMO}$ norm, which is one
more instance of the principle, running throughout this paper, that it is the localization to compact supports, and not any intrinsic
property of $\mathrm{VMO}$ alone, that makes the construction work. This proposition is the density invoked at the end of the proof of
Theorem \ref{thm:h_dist_vmo} to extend $D$ by continuity from $\mathrm{L}^\infty(\Rd)\cap\mathrm{VMO}_c(\Rd)$ to all of $\mathrm{VMO}_c(\Rd)$.
\end{remark}

\subsection{Strict Density of Smooth Test Functions}
In classical harmonic analysis, the space of smooth test functions $\mathrm{C}_c^\infty(\Rd)$ is not dense in $\mathrm{L}^\infty(\Rd)$
or $\mathrm{BMO}(\Rd)$, which obstructs standard smooth approximation arguments. The LF-space topology used here resolves this obstruction.

\begin{proposition}
The space of smooth functions with compact support, $\mathcal{D}(\Rd) = \mathrm{C}_c^\infty(\Rd)$, is dense in $\mathrm{VMO}_c(\Rd)$
under the inductive limit topology. In particular, it is dense, in the $\mathrm{BMO}$ topology, in the bounded subspace $\mathrm{L}^\infty(\Rd)\cap\mathrm{VMO}_c(\Rd)$.
\end{proposition}

\begin{proof}
Let $f \in \mathrm{VMO}_c(\Rd)$. By definition, there exists a compact set $K \subset \Rd$ such that $\mathrm{supp}(f) \subseteq K$,
meaning $f$ belongs to the Banach step space $\mathrm{VMO}_K(\Rd)$. 

Let $\eta_\epsilon$ be a standard smooth mollifier. The convolution $f_\epsilon = f \ast \eta_\epsilon$ belongs to $\mathrm{C}_c^\infty(K_\epsilon)$,
where $K_\epsilon$ is the $\epsilon$-neighborhood of $K$. Because $K_\epsilon$ is strictly contained within a slightly larger compact set $K'$ for
all sufficiently small $\epsilon$, the sequence $(f_\epsilon)$ is entirely contained within the step space $\mathrm{VMO}_{K'}(\Rd)$.

By the characterization of $\mathrm{VMO}(\Rd)$ as the closure of $\mathrm{C}_c(\Rd)$ in the $\mathrm{BMO}$ norm, the mollified sequence converges
to the original function in the $\mathrm{BMO}$ norm: $\|f_\epsilon - f\|_{\mathrm{BMO}(\Rd)} \to 0$ as $\epsilon \to 0$. Since the $\mathrm{BMO}$
norm is the norm of the Banach step $\mathrm{VMO}_{K'}(\Rd)$, and $\mathrm{VMO}_{K'}(\Rd)$ is a generating subspace of the strict inductive limit,
the sequence $(f_\epsilon)$ converges to $f$ in the topology of $\mathrm{VMO}_c(\Rd)$.

Finally, if in addition $f \in \mathrm{L}^\infty(\Rd)$, then standard mollifier properties give
$\|f_\epsilon\|_{\mathrm{L}^\infty(\Rd)} \leq \|f\|_{\mathrm{L}^\infty(\Rd)}$, so the approximating sequence stays within the bounded subspace
and converges to $f$ in the $\mathrm{BMO}$ topology. We stress that this is convergence in the inductive limit ($\mathrm{BMO}$-based) topology only;
the mollification need not converge in the $\mathrm{L}^\infty$ norm, and no such convergence is claimed or required.
\end{proof}

\begin{remark}
This density is a useful tool. It allows multiplier bounds, localization principles, or fractional defect representations to be proven first
on classical smooth Schwartz functions, and then passed to the limit in $\mathrm{VMO}_c(\Rd)$ via density, bridging the gap between rough
$\mathrm{VMO}$ coefficients and smooth analytical techniques.
\end{remark}

\subsection{The Fr\'echet Dual and Local Hardy Spaces}
The topological dual of the global space $\mathrm{VMO}(\Rd)$ is the Hardy space $\mathcal{H}^1(\Rd)$. Restricting the test functions to the
compactly supported LF-space changes the dual space, which matters for the existence of non-trivial macroscopic defect measures.

\begin{proposition}\label{prop:frechet_dual}
The topological dual space of the strict LF-space $\mathrm{VMO}_c(\Rd)$ is a Fr\'echet space characterized by local Hardy space distributions.
\end{proposition}

\begin{proof}
By the standard duality theory for inductive limits, the topological dual of a strict inductive limit of Banach spaces
$\mathrm{VMO}_c(\Rd) = \bigcup_n \mathrm{VMO}_{K_n}(\Rd)$ is topologically isomorphic to the projective limit of the dual spaces
$\mathrm{VMO}_{K_n}(\Rd)'$. Therefore, the dual space $\mathrm{VMO}_c(\Rd)'$ is a Fr\'echet space.

A linear functional $T$ acts continuously on this LF-space if and only if its restriction to every Banach step $\mathrm{VMO}_{K}(\Rd)$ is
continuous. The dual of the global $\mathrm{VMO}$ space is the classical Hardy space $\mathcal{H}^1(\Rd)$, whose elements are characterized
by specific cancellation properties (i.e., vanishing integral moments).

Because the functional $T$ is only required to be bounded on fixed compact supports $K$, it is not subject to the global moment condition
and acts as an element of the local Hardy space $h^1(\Rd)$ in the sense of Goldberg. 
\end{proof}

\begin{remark}
This duality clarifies why the $\mathrm{VMO}$ H-distribution framework is non-degenerate. Elements of the global Hardy space $\mathcal{H}^1(\Rd)$
must integrate to zero, so a functional constrained to the global Hardy space would satisfy a rigid cancellation condition. Because the LF-space
topology shifts the dual to a Fr\'echet space of \emph{local} Hardy distributions, the H-distributions are not subject to a global zero-integral
condition and can capture non-trivial localized microlocal concentrations. We caution that, unlike classical H-measures, these functionals are in
general signed (indeed complex-valued) and carry no positivity; the physical interpretation is that of a localized microlocal density of possibly
indefinite sign, not a non-negative energy measure.
\end{remark}

\subsection{Joint Continuity of Bilinear Functionals}
In Theorem \ref{thm:h_dist_vmo}, the H-distribution is defined as a bilinear functional acting on the product space
$\big(\mathrm{L}^\infty(\Rd)\cap\mathrm{VMO}_c(\Rd)\big) \times \mathrm{C}^\kappa(\mathrm{S}^{d-1})$. Establishing continuity on product
topologies involving inductive limits requires care, as separate continuity does not always guarantee joint continuity.
The strict LF-structure resolves this obstruction.

\begin{proposition}\label{prop:joint_cont}
Let $Y = \mathrm{C}^\kappa(\mathrm{S}^{d-1})$ be the Banach space of frequency symbols. A bilinear functional 
\begin{equation*}
D: \mathrm{VMO}_c(\Rd) \times Y \to \mathbb{C}
\end{equation*}
is jointly continuous if and only if for every compact set $K \subset \Rd$, its restriction to the Banach step space $\mathrm{VMO}_K(\Rd) \times Y$ is bounded.
\end{proposition}

\begin{proof}
The space $\mathrm{VMO}_c(\Rd)$ is a strict LF-space. By the Baire category theorem, Fr\'echet spaces and strict
inductive limits of Fr\'echet spaces are barrelled \cite{NB}. 

By the universal property of inductive limits, a mapping from an LF-space into a locally convex space is continuous
if and only if its restriction to every generating step space is continuous. For a bilinear mapping evaluated on
$\mathrm{VMO}_K(\Rd) \times Y$, because both $\mathrm{VMO}_K(\Rd)$ and $Y$ are complete Banach spaces, joint continuity
is strictly equivalent to the existence of a uniform bounding constant $C_K > 0$ such that:
\begin{equation*}
|D(\varphi, \psi)| \leq C_K \|\varphi\|_{\mathrm{BMO}(\Rd)} \|\psi\|_Y.
\end{equation*}
Because the LF-space is barrelled and $Y$ is a Banach space, the Banach-Steinhaus theorem ensures that establishing this
bounded continuity on every individual step space is exactly the necessary and sufficient condition to guarantee joint
continuity on the global product space topology.
\end{proof}

\begin{remark}
This proposition justifies the functional bounds derived in Theorem \ref{thm:h_dist_vmo}. Because the localized John-Nirenberg
inequality establishes the existence of the bounding constant $C_{K_\varphi}$ for any fixed compact support $K_\varphi$,
the H-distribution $D$ is a jointly continuous bilinear functional on the chosen product topology.
\end{remark}

\subsection{The Projective Tensor Product and Operator Representation}
Because the H-distribution is a jointly continuous bilinear functional on our product topology, it factors, by the universal
property of the projective tensor product, through a single continuous linear functional. This gives an operator-theoretic
view of the macroscopic defect measure; it is analogous in spirit to the kernel representation of classical distributions,
though, both factors being ordinary locally convex spaces, it requires only the universal property and none of the nuclearity
underlying the Schwartz Kernel Theorem.

\begin{proposition}\label{prop:operator_rep}
Let $X = \mathrm{VMO}_c(\Rd)$ and $Y = \mathrm{C}^\kappa(\mathrm{S}^{d-1})$. The jointly continuous H-distribution $D$ extends
uniquely to a continuous linear functional on the projective topological tensor product $X \hat{\otimes}_\pi Y$.
Furthermore, $D$ induces a continuous linear operator from the strict LF-space $X$ into the space of spherical distributions.
\end{proposition}

\begin{proof}
Let $D: X \times Y \to \mathbb{C}$ be the jointly continuous bilinear functional established in the preceding proposition.
By the universal property of the projective topological tensor product, any jointly continuous bilinear mapping defined on
the Cartesian product of two locally convex spaces factors uniquely through a continuous linear mapping on their projective
tensor product $X \hat{\otimes}_\pi Y$ \cite{Treves}. Thus, the H-distribution is identified as a single continuous linear
functional belonging to the dual space $(X \hat{\otimes}_\pi Y)'$.

Additionally, this joint continuity guarantees that the partial evaluation mapping $\varphi \mapsto D(\varphi, \cdot)$ is a
well-defined, continuous linear operator from $X$ into the strong topological dual $Y'$. Since $\kappa$ is the least integer
strictly greater than $d/2$, $Y = \mathrm{C}^\kappa(\mathrm{S}^{d-1})$ is the Banach space of $\kappa$-times continuously
differentiable functions on the sphere, and its strong dual $Y'$ is precisely the space of \emph{spherical distributions of order at most $\kappa$}
on the compact manifold $\mathrm{S}^{d-1}$, that is, the continuous linear functionals on $\mathrm{C}^\kappa(\mathrm{S}^{d-1})$.
\end{proof}

\begin{remark}
This operator representation reads the H-distribution as a measuring device. Given a localized spatial test function $\varphi \in X$,
which isolates a geometric region of highly oscillatory $\mathrm{VMO}$ concentration, the continuous linear operator returns the angular
spectrum of the trapped macroscopic energy as a spherical distribution. Thus the defect measure separates the spatial location of the
trapped energy from its directional scattering signature. This operator representation is used directly in the main text:
Corollary \ref{cor:angular_defect} uses it to construct the angular defect distribution $T_{Q\phi} = D(Q\phi,\cdot)$ and to identify
the scalar defect of Theorem \ref{thm:defect_rep} as its constant-symbol component.
\end{remark}

\subsection{Separability and Sequential Weak-$\ast$ Compactness}
In classical functional analysis, the space $\mathrm{L}^\infty(\Rd)$ is non-separable. Consequently, bounded sets in its topological
dual are not metrizable in the weak-$\ast$ topology, which prevents the extraction of weakly-$\ast$ convergent subsequences and forces
the use of generalized nets. The topology of our test space avoids this obstruction.

\begin{proposition}\label{prop:separable}
The strict LF-space $\mathrm{VMO}_c(\Rd)$ is separable, and so is its bounded subspace $\mathrm{L}^\infty(\Rd) \cap \mathrm{VMO}_c(\Rd)$.
Consequently, every bounded sequence in the Fr\'echet dual space admits a weakly-$\ast$ convergent subsequence.
\end{proposition}

\begin{proof}
Each Banach step $\mathrm{VMO}_{K_n}(\Rd)$ is separable: the global space $\mathrm{VMO}(\Rd)$ is separable, being the closure of
$\mathrm{C}_c(\Rd)$ in the $\mathrm{BMO}$ norm, and separability passes to its closed subspaces. By the density proposition above,
$\mathcal{D}(\Rd)$ is dense in $\mathrm{VMO}_c(\Rd)$ in the inductive limit topology; since $\mathcal{D}(\Rd)$ is itself separable it
contains a countable dense subset $\mathcal{M}$, and the continuity of the inclusion $\mathcal{D}(\Rd)\hookrightarrow\mathrm{VMO}_c(\Rd)$
forces $\mathcal{M}$ to be dense in $\mathrm{VMO}_c(\Rd)$ as well. Thus $\mathrm{VMO}_c(\Rd)$ is a separable topological vector space,
and its subspace $\mathrm{L}^\infty(\Rd)\cap\mathrm{VMO}_c(\Rd)$ is separable in the induced ($\mathrm{BMO}$-based) topology.

By the generalized Banach-Alaoglu theorem, the polar of any neighborhood of the origin in a separable locally convex space is a
weakly-$\ast$ compact, metrizable subset of the dual space. Therefore, bounded sets in the Fr\'echet dual space are sequentially
compact in the weak-$\ast$ topology, guaranteeing that any bounded sequence of linear functionals admits a weakly-$\ast$ convergent subsequence.
\end{proof}

\begin{remark}
This separability justifies the Cantor diagonal extraction used in the proof of Theorem \ref{thm:h_dist_vmo}. We note that the argument
does not require separability of $\mathrm{L}^\infty(\Rd)$: it is the $\mathrm{BMO}$-based inductive limit topology of $\mathrm{VMO}_c(\Rd)$
that is separable, and the essential supremum enters only as a set restriction on the functions being evaluated, never as part of the topology.
\end{remark}

\end{document}